\documentclass[10pt]{article}

\usepackage{amssymb,amsthm,amsmath,hyperref}

\numberwithin{equation}{section}
\usepackage{cases}
  \usepackage{paralist}
  \usepackage{graphics} 
  \usepackage{epsfig} 
\usepackage{graphicx}  \usepackage{epstopdf}
\numberwithin{equation}{section}

\newtheorem{thm}{Theorem}[section]
\newtheorem{lem}[thm]{Lemma}

\newtheorem{claim}{Claim}[section]
\newtheorem{theorem}{Theorem}[section]

\newtheorem{prop}[theorem]{Proposition}

\theoremstyle{definition}
\newtheorem{defn}[theorem]{Definition}
\newtheorem{rem}[theorem]{Remark}

\allowdisplaybreaks

\begin{document}
\title{ Modified scattering  for the one-dimensional Schr\"odinger equation with a subcritical dissipative nonlinearity}

\author{Xuan Liu\footnote{E-mail adress: lxmath@zju.edu.cn}, Ting Zhang\footnote{E-mail adress: zhangting79@zju.edu.cn}\\
	\small{School of Mathematical Sciences, Zhejiang University,
		Hangzhou 310027, China}}
\date{}

\maketitle


\begin{abstract}
 We study the asymptotic behavior in time of solutions to the one dimensional nonlinear Schr\"odinger equation with a subcritical  dissipative nonlinearity $\lambda |u|^\alpha u$, where $0<\alpha<2$, and $\lambda $ is a complex constant satisfying $\text{Im} \lambda >\frac{\alpha |\text{Re} \lambda |}{2\sqrt{ \alpha +1}}$. For arbitrary large initial data,  we present the  uniform time decay estimates  when $4/3<\alpha <2$, and the large time asymptotics of the solution when $\frac{7+\sqrt{145}}{12}<\alpha <2$.  The proof is based on the vector fields method and   a semiclassical analysis method. 
\end{abstract}
\section{Introduction}
\textbf{Background and historical notes.} We consider the large time behavior of solutions to the  one-dimensional Schr\"{o}dinger equation
\begin{eqnarray}\label{NLS}
	\begin{cases}
		i\partial_t  u+\frac{1}{2}\partial_x^2u+\lambda |u|^\alpha u=0,\ t>0,\ x\in \mathbb{R},\\
		u(0,x)=u_0(x),\ x\in \mathbb{R},
	\end{cases}
\end{eqnarray}
where $\alpha >0$, $\lambda \in\mathbb{C}$, and  $u:\mathbb{R}\times\mathbb{R}\rightarrow\mathbb{C}$ is a complex-valued function.  From the physical point of view, (\ref{NLS}) is said to be a governing equation of the light traveling through optical fibers, in which  $|u(t,x)|$ describes the amplitude of the electric field,  $t$ denotes the position along the fiber and  $x$ stands for the temporal parameter expressing a form of pulse. As for the nonlinear coefficient,  $\text{Re}\lambda  $ denotes the magnitude of the nonlinear Kerr effect and  $\text{Im} \lambda $ implies the magnitude of dissipation due to nonlinear Ohm's law (see e.g. \cite{GP}). The equation (\ref{NLS}) is also  a particular case of the  more general complex Ginzburg-Landau equation on $\mathbb{R}^N$
\begin{eqnarray}
	\partial_{t} u=e^{i \theta} \Delta u+\zeta|u|^{\alpha} u\notag
\end{eqnarray}
where $|\theta| \leq \frac{\pi}{2}$ and $\zeta \in \mathbb{C}$, which   is a generic modulation equation that describes the nonlinear evolution of patterns at near-critical conditions. See for instance \cite{Cr,Mie,Ste}.

There are many papers that studied the global well-posedness problem, decay and the asymptotic behavior of the solution (see \cite{CaHan,CaJFA,DZ03,Hayashi,Kita,Kita2,Ozawa,Zhang} and references therein).   We use the following classification with respect to the value $\alpha $: the values $\alpha >2$ we call the super-critical in the scattering problem, the value $\alpha =2$ is the critical one and $0<\alpha <2$ we refer to  the sub-critical. 

Let us  recall some known large time asymptotics of (\ref{NLS}) with  $\lambda \in \mathbb{R}$. Concerning  the super-critical case $\alpha >2$, it is well-known that the solution $u(t)$ behaves like a free solution $\exp \left((it/2)\partial_  x^2\right)\phi$ for $t$ sufficiently large (\cite{Ginibre0, Strauss2, Tsutsumi}). The strategy for this free asymptotic profile  largely relies on the rapid decay of the nonlinearity. More precisely, since $\int_{1}^{\infty }|u(t)|^\alpha  \mathrm{d}t\approx\int_{1}^{\infty }t^{-\alpha /2} \mathrm{d}t<\infty $ by expecting that $u(t)$ decays like a free solution, the nonlinearity can be regarded as negligible in the long time  dynamics. As for  the sub-critical and critical case  $0<\alpha \le2$, the situation changes.  The nonexistence of usual scattering states was obtained  \cite{Ba,Strauss} by making use of the time decay estimate of solutions  obtained from  pseudo-conformal conservation law.   In the case  $\alpha =2$, Ozawa \cite{Ozawa} constructed modified wave operators to the equation (\ref{NLS}) for small scattering states, and Hayashi-Naumkin \cite{Hayashi} proved the time decay and the large time asymptotics of  $u(t)$ for small initial data. According to their results,   the small solution  $u(t)$ asymptotically tends to a modified free solution. More precisely, there are  $\mathbb{C}$-valued $W(x)\in L^\infty \cap L^2$ and  $\mathbb{R}$-valued  $\Phi(x)\in L^\infty $ such that as  $t\rightarrow \infty $ 
	\begin{equation}
	u(t,x)= \frac{1}{\sqrt {it}}W(\frac{x}{t})\exp \left(i\frac{x^2}{2t}+i\lambda |W(\frac{x}{t})|^2\log t+i\Phi(\frac{x}{t})\right)+O_{L^\infty _x}(t^{-1/2})\notag
\end{equation}
While the subcritical case  $0<\alpha <2$ seems to be completely open. To our knowledge, there are no results about the precise behavior or any kind of modified scattering of  $u$ for large time.

Many works have also dealt with the complex coefficient case. We can only expect the large time asymptotics of (\ref{NLS}) for the case  $\text{Im} \lambda >0$, since it is proved in \cite{Ca2} that, under the assumptions   $\text{Im} \lambda <0,\ 0<\alpha <\infty $, there exists a class of  blowup solutions to (\ref{NLS}). On the other hand,  it is easy to see  that 
\begin{equation}
	\|u(t)\|_{L^2}^2+2\text{Im} \lambda \int _0^t \|u(\tau)\|_{L^{\alpha +2}}^{\alpha +2}d\tau =\|u_0\|_{L^2}^2,\label{141}
\end{equation}
which suggests a dissipative structure for  $\text{Im} \lambda >0$.   In what follows,  we  concentrate our attention on  the sub-critical and critical  case: $0<\alpha \le2$  (For the  super-critical case, we refer to the aforementioned  works  \cite{Ginibre0, Strauss2, Tsutsumi}, where a  super-critical  real $\lambda $ were studied, and  the ideals are still applicable to a complex $\lambda $ with  $\text{Im} \lambda >0$).   The  critical case $\alpha =2$ has been studied in \cite{CPDE}, in which the positivity of $\text{Im} \lambda $ visibly affects the decay rate of $\left\|u(t,x)\right\|_{L^\infty _x}$ and, actually, it decays like $(t\log t)^{-1/2}$.  Since the nontrivial free solution only decays like  $O(t^{-1/2})$, this gain of additional logarithmic time decay reflects a dissipative character.  This result is then extended in \cite{Kita,Kita2}  to the subcritical  case. For   $\alpha <2$ is sufficiently close to  $2$,  Kita-Shimomura  \cite{Kita} established the time decay estimates 
\begin{equation}
		\|u(t,x)\|_{L^\infty _x}\lesssim t^{-1/\alpha },\qquad \text{for }t\ge1 \label{dsa}
\end{equation}
and the asymptotic formula of the solutions. In addition,  it is proved  in \cite{Kita2} that, 
 under the large dissipative assumption 
\begin{eqnarray}\label{lambda1}
	\lambda _2\ge\frac{\alpha\left|\lambda _1\right| }{2\sqrt{ \alpha +1}},\qquad \lambda =\lambda _1+i\lambda _2,
\end{eqnarray} 
all solutions with initial value in  $H^1(\mathbb{R})\cap L^2 (\mathbb{R}, |x|^2dx)$ satisfy the  $L^\infty $ decay estimate (\ref{dsa})  when $\frac{1+\sqrt {33}}{4}<\alpha <2$,  and  possess a  large time asymptotic  state when  $\frac{9+\sqrt {177}}{12}<\alpha <2$.  
The strategy used  in \cite{Kita, Kita2,CPDE} is to apply  the operator  $\mathcal{F} U(-t)$ to the equation (\ref{NLS}), where  $U(t)=e^{it/2 \Delta }$ is the Schr\"odinger operator. Using the  factorization technique of the Schr\"odinger operator  $U(t)$, they obtain an ODE  for  $\mathcal{F} U(-t)u(t)$
\begin{equation}
	i\partial_t  \mathcal{F} U(-t)u(t)=\lambda t^{-\alpha /2}|\mathcal{F} U(-t)u(t)|^\alpha \mathcal{F} U(-t)u(t)+O_{L^\infty_x } (t^{-\alpha /2-\mu}),\ 0<\mu<1/4.,\label{191}
\end{equation}
from which, they deduce the large time asymptotics of  $\mathcal{F} U(-t)u(t,x)$ and then in the solution  $u(t,x)$.


The present work aims  to complete the previous results on the time  asymptotic behavior of the solutions obtained in \cite{Kita,Kita2,CPDE}.  More precisely, for arbitrary large initial data,  we present the  uniform time decay estimates  when $4/3<\alpha <2$, and the large time asymptotics of the solution when $\frac{7+\sqrt{145}}{12}<\alpha <2$.
\vspace{0.5cm} 

\noindent \textbf{Notation and function spaces.} To state our result precisely, we now give some notations. Throughout the paper, $F(\xi)$ denotes the  second order constant coefficients classical elliptic symbol, which has  an expansion
\begin{eqnarray}
	\label{1.3}F(\xi)=c_2\xi^2+c_1\xi+c_0
\end{eqnarray}
with $c_2>0,\ c_1,\ c_0\in \mathbb{R}$. We introduce the notations $D_t=\frac{\partial_t}{i}$, $D=\frac{\partial_x}{i}$ and  the vector field 
\begin{equation}
	\mathcal{L} =x+tF'(D).\label{5161}
\end{equation}
For  $\psi \in L^1(\mathbb{R})$,  $\mathcal{F} \psi$  is represented as  $\mathcal{F} \psi (\xi)=(2\pi)^{-1/2}\int_{ \mathbb{R}}\psi(x) e^{-ix\xi}dx$.   $[A,B]$ denotes the commutator  $AB-BA$. Different positive constants we denote by the same letter $C$.   We introduce some function spaces.  $\mathcal {S} (\mathbb{R}^2) $ denotes the usual two-dimensional Schwarz space.     $L^p=L^p(\mathbb{R})$  denotes the usual Lebesgue space with the norm  $\|\phi\|_{L^p}=(\int _{\mathbb{R}}|\phi(x)|^p dx)^{1/p}$ if  $1\le p<\infty $ and  $\|\phi\|_{L^\infty }=\text{ess. sup }\left\{|\phi(x)|;x\in \mathbb{R}\right\} $. The weighted Sobolev space is defined by  $	H^{0,m}=L^2(\mathbb{R})\cap L^2(\mathbb{R},|x|^{2m}dx)$. 
\vspace{0.5cm} 

\noindent \textbf{Main results.} We are now ready to state the main result.
\begin{thm}\label{T1.1}
	Assume  that $u_0\in H^{0,1}$, $0<\alpha <2$,  $\lambda $ satisfies the condition (\ref{lambda1}) and $\mathcal{L} $ is the vector filed defined in (\ref{5161}). Then there exists a unique global solution $u\in C\left([0,\infty  ),\ L^2 \right)$ to the Cauchy problem 
	\begin{equation}\left\{
		\begin{array}{ll}
			(D_t -F(D))u=\lambda|u|^\alpha u,&t>0,x\in\mathbb{R}
			\\ u(x,0)= u_0(x),
		\end{array}\label{1.1}
		\right.
	\end{equation}
satisfying  
	\begin{eqnarray}\label{3192}
		\left\|u(t,x)\right\|_{L_x^2}+\left\|\mathcal{L}u(t,x)\right\|_{L_x^2}\le C\left\|u_0\right\|_{ H^{0,1}},\ t\ge0 
	\end{eqnarray}
	and 
	\begin{equation}
		\label{ldecay}
		\|u(t,x)\|_{L^\infty _x}\le C\|u_0\|_{H^{0,1}}t^{-1/2},\ t>1.
	\end{equation}
	Furthermore, if $u_0\in H^{0,2}$ and $\alpha \ge1$,  then
	\begin{eqnarray}\label{491}
		\left\|\mathcal{L} ^2u(t,x)\right\|_{L_x^2}\le C(\left\|u_0\right\|_{H^{0,2}}+\left\|u_0\right\|_{H^{0,2}}^{2\alpha +1})t^{2-\alpha }, \ t\ge0.
	\end{eqnarray}
\end{thm}
\begin{rem}
	Applying the operator  $\mathcal{L} $ to the equation (\ref{1.1}), we obtain easily  the energy inequality 
	\begin{equation}
		\|\mathcal{L} u(t,\cdot)\|_{L^2}\lesssim  \|\mathcal{L} u(1,\cdot)\|_{L^2}+\int_{1}^{t}\|u(\tau,\cdot)\|_{L^\infty }^\alpha \|\mathcal{L} u(t,\cdot)\|_{L^2} \mathrm{d}\tau .\label{energy}
	\end{equation}
Using Gronwall's inequality and a priori estimate  $\|u(\tau,\cdot)\|_{L^\infty }\lesssim \varepsilon \tau^{-1/2} $, Hayashi and Naumkin \cite{Hayashi} obtained a  moderate growth rate  of  $\|\mathcal{L} u\|_{L^2}$ in the critical case $\alpha =2$, which is essential to close the bootstrap assumption on  $\|u(t,x)\|_{L^\infty _x}.$  
	While the limit in [Theorem \ref{T1.3}, part (d)] demonstrates that one can not hope to deduce this from  (\ref{energy})   for  $\alpha <2$ as in \cite{Hayashi}, even for the  small initial data.  So  we assume the large dissipative condition (\ref{lambda1}) as in \cite{Kita2}, which is used in (\ref{151}) to derive the uniform bound (\ref{3192}). 
\end{rem}
Next,  we derive  the time decay rate of  the global solution obtained in Theorem \ref{T1.1}. 
\begin{thm}
	\label{T1.2}
	Assume that $u_0\in H^{0,1}$, $4/3<\alpha<2$, $\lambda $ satisfies the condition (\ref{lambda1}) and $u$ is the global solution obtained in Theorem \ref{T1.1}.  There exists a constant $C>0$ such that for all $t\ge1$,
	\begin{eqnarray}\label{decay}
		\left\|u(t,x)\right\|_{L_x^\infty }\le Ct^{-1/\alpha }.
	\end{eqnarray}
\end{thm}
\begin{rem}
	A similar time decay estimate as in   (\ref{decay}) was obtained in \cite{Kita2} under the assumptions $\frac{1+\sqrt{33}}{4}<\alpha \le2$, which was then extended to the case  $\frac{7+\sqrt{145}}{12}<\alpha \le2$ in \cite{Jin}.  We note that $\frac{4}{3}<\frac{1+\sqrt{33}}{4}\approx 1.686$ and $\frac{4}{3}<\frac{7+\sqrt{145}}{12}\approx1.586$. Therefore, Theorem \ref{T1.2} is an improvement of the corresponding results in \cite{Jin,Kita2}. 
\end{rem}
\begin{rem}
The additional assumption  $\alpha >4/3$ ensures that the remainder term    $v_{\Lambda ^c}$ decays faster than   $v$  (see (\ref{1511}) and (\ref{1510})). 
\end{rem}
\begin{rem}
	Theorem \ref{T1.2} is valid without any  smallness conditions on the initial data. Moreover, the solution decays faster than the free solution. Recall that in one space dimension, the free solution decays like $t^{-1/2}$. 
\end{rem}

Finally, we give a large time  asymptotic formula for the solutions and show the existence of modified scattering states for a certain range of the exponent in the nonlinear term.  
\begin{thm}\label{T1.3}
Suppose that the assumptions in Theorem \ref{T1.2} are satisfied and 
	\begin{equation}
		u_0\in H^{0,1},\ \frac{1+\sqrt{33}}{4}<\alpha <2,\notag
	\end{equation}
	or 
	\begin{equation}
		u_0\in H^{0,2},\ \frac{7+\sqrt{145}}{12}<\alpha <2,\notag
	\end{equation}
	then  the followings hold:\\
	(a) Let  
	\begin{equation}\label{eq1}
		\Phi(t,x)=\int_{1}^{t}s^{-\alpha /2}\left|v_\Lambda(s,x) \right| ^\alpha  \mathrm{d}s,
	\end{equation}
	where  $v_{\Lambda }$ is the function defined in (\ref{1219}). There exists a unique complex valued function $z_+(x)\in L^\infty _x\cap L^2_x$ such that for some $\kappa>0$,
	\begin{eqnarray}
		\left\|v_{\Lambda }(t,x)\exp\left(-i(w(x)t+\lambda \Phi(t,x))\right)-z_+(x)\right\|_{L^\infty _x\cap L^2_x}=O(t^{-\kappa})\notag
	\end{eqnarray}
	holds as $t\rightarrow \infty $, where  $w(x)=:-(x+c_1)^2/(4c_2)+c_0.$ \\
	(b) Let 
	\begin{eqnarray}\label{3301}
		K(t,x)=1+\frac{2\alpha \lambda _2}{2-\alpha }\left|z_+(x)\right|^\alpha \left(t^{(2-\alpha )/2}-1\right),
	\end{eqnarray}
	\begin{eqnarray}
		\psi_+(x)=\alpha \lambda_2 \int_{1}^{\infty }s^{-\alpha /2}\left(\left|v_{\Lambda }(s,x)\right|^\alpha\exp \left( {\alpha \lambda _2\Phi(s,x)}\right) -\left|z_+(x)\right|^\alpha \right) \mathrm{d}s,\label{a2}
	\end{eqnarray}
	and 
	\begin{equation}\label{eq2}
		S(t,x)=\frac{1}{\alpha \lambda _2}\log \left(K(t,x)+\psi_+(x)\right).
	\end{equation}
	The asymptotic formula 
	\begin{eqnarray}\label{3302}
		u(t,x)=\frac{1}{\sqrt t}e^{i\left(w(\frac{x}{t})t+\lambda S(t,\frac{x}{t})\right)}z_+(\frac{x}{t})+O_{L^\infty _x}(t^{-1/2-\kappa})\cap O_{L^2_x}(t^{-\kappa})
	\end{eqnarray}
	holds 	as $t\rightarrow \infty $, where $\kappa$ is the same  constant as in part (a). \\
	(c) Let $u_+(x)=\frac{1}{\sqrt {4\pi c_2}} e^{-i\frac{\pi}{4}}e^{-i\frac{c_1x}{2c_2}} (\mathcal{F} z_+)(\frac{x}{2c_2})$, we have the modified linear scattering 
	\begin{equation}
		\lim_{t\rightarrow \infty }\left\|u(t,x)-e^{i\lambda S(t,\frac{x}{t})}e^{iF(D)t}u_+(x)\right\|_{L_x^2}=0.\label{123}
	\end{equation}\\
	(d) If $u_0\neq0$, then the limit
	\begin{eqnarray}
		\lim_{t\rightarrow \infty } t^{\frac{1}{\alpha }}\|u(t,x)\|_{L^\infty _x}=\left(\frac{2-\alpha }{2\alpha \lambda _2}\right)^{\frac{1}{\alpha }},\qquad \text{when }\alpha _0<\alpha <2,\label{qwe2}
	\end{eqnarray}
	exists and is independent of the initial value, where  $\alpha _0=\frac{5+\sqrt {89}}{8} \approx 1.804$. 
\end{thm}
\begin{rem}
	A similar large time asymptotic formula of the solutions is obtained in \cite{Kita2} in the case $\frac{9+\sqrt{177}}{12}<\alpha <2$. 
	Since $\frac{7+\sqrt{145}}{12}\approx 1.587< \frac{1+\sqrt{33}}{4}\approx 1.686<\frac{9+\sqrt{177}}{12}\approx 1.859$, we see that Theorem \ref{T1.3} generates the result of \cite{Kita2} in the range of $\alpha $. 
\end{rem}
\begin{rem}
	The  assumption on the lower bound of  $\alpha $  ensures the convergence of the integral (\ref{3222}). It can be extended to  $\alpha >\frac{2}{3}$ if we assume some nonvanishing conditions on the initial values. See e.g. \cite{CaHan,CaJFA}. 
\end{rem}
\begin{rem}
	According to the asymptotic formulas (\ref{3302}) and (\ref{123}), we see that the solution  $u$ is not asymptotically free. 	By the definition (\ref{eq2}) of $S(t,x)$, we can write the modification factor $e^{i\lambda S(t,x)}$ explicitly:
	\begin{equation}
		e^{i\lambda S(t,x)}= \frac{\exp \left\{\frac{i\lambda _1}{\alpha \lambda _2}\log \left\{1+\frac{2\alpha \lambda _2}{2-\alpha }|z_+(x)|^\alpha (t^{(2-\alpha )/2}-1)+\psi_+(x)\right\} \right\} }{(1+\frac{2\alpha \lambda _2}{2-\alpha }|z_+(x)|^\alpha (t^{(2-\alpha )/2}-1)+\psi_+(x))^{1/\alpha }}.\label{652}
	\end{equation}
\end{rem}

\vspace{0.5cm}
 
\noindent \textbf{Strategy of the proof.} We briefly sketch the strategy used to derive the  decay  estimate (\ref{decay}), which  is the key to establishing the large time asymptotics of the solution.      We adapt the  semiclassical analysis method  introduced  by Delort  \cite{Delort}, see also \cite{S,Zhang} which are more close to the problem we are considering. We make first a semiclassical change of variables
\begin{equation}
	u(t,x)=\frac{1}{\sqrt{t}}v(t,\frac{x}{t}),\label{uv1}
\end{equation}
for some new unknown function $v$, that allows to  rewrite   the equation  (\ref{1.1})    as
\begin{equation}\label{1.15}
	(D_t-G_h^w(x\xi+F(\xi)))v=\lambda h^{\alpha /2}|v|^\alpha v,
\end{equation}
where the semiclassical parameter $h=\frac1t$, and the Weyl quantization of a symbol $a$  is  given by 
\begin{equation}
	G_h^w(a)u(x)=\frac{1}{2\pi h}\iint e^{\frac{i}{h}(x-y)\xi}a(\frac{x+y}{2},\xi)u(y)dyd\xi. \notag
\end{equation}
By (\ref{uv1}), the decay estimate (\ref{decay}) is equivalent to 
\begin{equation}
	\|v(t,x)\|_{L^\infty _x}\le Ct^{1/2-1/\alpha }.\label{1511}
\end{equation}
If we develop the symbol  $x\xi+F(\xi)$ as follows by using (\ref{1.3})
\begin{equation}
	x\xi+F(\xi)=w(x)+\frac{(x+F'(\xi))^2}{4c_2}\qquad \text{with  } w(x)=-\frac{(x+c_1)^2}{4c_2}+c_0,\label{15100}
\end{equation}
we deduce from  (\ref{1.15}) an ODE for  $v$: 
\begin{equation}
	D_tv=w(x)v+\lambda h^{\alpha /2}|v|^\alpha v+\frac{1}{4c_2}G_h^w((x+F'(\xi))^2)v\label{e123}
\end{equation}
By semiclassical Sobolev inequality,  $\|G_h^w((x+F'(\xi))^2)v\|_{L^\infty _x}$ is  controlled by  some energy norm of  $v$ that contains  the spatial derivative of order three. While deducing  this norm  from the  equation (\ref{1.15}) via  the standard energy method  requires   $\alpha \ge 2$. Instead, we use the operators whose symbols are localized  in a neighbourhood of $M=:\{(x,\xi)\in \mathbb{R}^2:\ x+F'(\xi)=0\}$ of size $O(\sqrt h)$.
In that way we can apply Proposition \ref{P2.6} to pass  uniform  norms of the remainders to the  $L^2$ norm losing only a power  $h^{-1/4}$. 

 More precisely, we  set  \begin{eqnarray}
	v_{\Lambda }=G_h^w(\gamma(\frac{x+F'(\xi)}{\sqrt h}))v,\label{1219}
\end{eqnarray}
where $\gamma\in C_0^\infty (\mathbb{R})$ satisfying $\gamma=1$ in a neighbourhood of zero.  In Lemma \ref{l4.3-00}, we will show that 
$v_{\Lambda ^c}=:G_h^w(1-\gamma(\frac{x+F'(\xi)}{\sqrt {h}}))v$ satisfies the uniform estimate
\begin{equation}
	\|v_{\Lambda ^c}(t,x)\|_{L^\infty _x}\lesssim t^{-1/4}.\label{1510}
\end{equation}
We see that it sufficies to prove the estimate 
\begin{equation}
	\|v_{\Lambda }(t,x)\|_{L^\infty _x} \le Ct^{1/2-1/\alpha },\notag
\end{equation}
since  $v_{\Lambda ^c}$ decays faster than  $v$  by the assumption  $\alpha >4/3$. 
Applying  $G_h^w(\gamma(\frac{x+F'(\xi)}{\sqrt{h}}))$ to (\ref{e123}) and using (\ref{15100}) we obtain the ODE for $v_{\Lambda }$ 
\begin{equation}
D_tv_\Lambda=w(x)v_{\Lambda }+\lambda h^{\alpha /2}  |v_\Lambda|^\alpha v_\Lambda+ R(v)\label{552}
\end{equation}
where the remainder 
\begin{eqnarray}
R(v)&=&[D_t-G^w_h(x\xi+F(\xi)),G^w_h(\gamma(\frac{x+F'(\xi)}{\sqrt{h}}))]v+\frac{1}{4c_2}G_h^w((x+F'(\xi))^2)v_{\Lambda }\notag\\
&&-\lambda h^{\alpha /2} G^w_h(1-\gamma(\frac{x+F'(\xi)}{\sqrt{h}}))(|v|^\alpha v)+\lambda h^{\alpha /2}\left( |v|^\alpha v-|v_\Lambda|^\alpha v_\Lambda\right)\notag
\end{eqnarray}
satisfies the estimate (see Lemmas \ref{l4.3}--\ref{l5.6})
\begin{equation}
	\|R(v)\|_{L^\infty _x}\lesssim t^{-5/4}+(\|v_{\Lambda }\|_{L^\infty _x}^\alpha +\|v\|_{L^\infty _x}^\alpha )t^{-\alpha /2-1/4}.\notag
\end{equation} 
Note that  $R(v)$  decays faster than the remainder in (\ref{191}) when  $\alpha <2$. Performing a bootstrap and a contradiction argument, one finally deduce from the ODE (\ref{552})  the desired  $L^\infty $ estimate for  $v_{\Lambda }$, and then in the solution  $u$.  
The details can be found in Subsection \ref{sub2}. 

\vspace{0.5cm} 
\noindent \textbf{Outline.} 
The framework of this paper is organized as follows.

 In Section \ref{S2}, we present the  definitions and  some useful properties of Semiclassical pseudo-differential operators.  In Section \ref{S3}, we establish the  global existence and uniqueness of the solution to (\ref{1.1}).  In Section \ref{S5}, we prove the decay estimates as stated in Theorem \ref{T1.2}, combining the  bootstrap and the contradiction argument.   Finally, in Section \ref{S6}, we  establish the asymptotic formulas in Theorem \ref{T1.3}. 
\section{Semiclassical pseudo-differential operators}\label{S2}
The proof of the main theorem will rely on the use of the  semiclassical pseudo-differential calculus. For simplicity, we give only the definitions and properties of the  operators we shall use. For more properties about semiclassical pseudo-differential operators, we refer to  Chapter 7 of the book of Dimassi-Sj\"{o}strand \cite{D-S} and  Chapter 4 of the book of Zworski \cite{Zworski}. 
\begin{defn}
	Let  $a(x,\xi)\in \mathcal{S} (\mathbb{R}^2)$ and  $h\in ( 0,1] $. Define the Weyl quantization to be the operator $G^w_h(a)$ acting on $u\in \mathcal{S}(\mathbb{R})$ by the formula
	\begin{equation}
		G^w_h(a)u=\frac{1}{2\pi h}\int_{\mathbb{R}}\int_{\mathbb{R}}
		e^{\frac{i}{h}(x-y)\xi}a(\frac{x+y}{2},\xi)u(y)dyd\xi.\notag
	\end{equation}
\end{defn}
We  have the following boundedness for Weyl quantization. 
\begin{prop}[Proposition 2.7 in \cite{Zhang}]\label{P2.6}
	Let $a(\xi)$ be a smooth function satisfying $|\partial_\xi^\alpha a(\xi)|$ $\leq C_\alpha <\xi>^{-1-\alpha}$ for any  $\alpha \in \mathbb{N}$. Then for  $h\in ( 0,1] $
	\begin{equation}
		\|G^w_h(a(\frac{x+F'(\xi)}{\sqrt{h}}))\|_{\mathcal{L}(L^2,L^\infty)}=O(h^{-\frac{1}{4}}),
		\ \|G^w_h(a(\frac{x+F'(\xi)}{\sqrt{h}}))\|_{\mathcal{L}(L^2,L^2)}=O(1).\notag
	\end{equation}
\end{prop}
Next, we introduce some useful composition  properties for  Weyl quantization. 
\begin{prop}[Theorem 7.3 in \cite{D-S}]\label{ab}
	Suppose that $a,b\in\mathcal{S}(\mathbb{R}^2)$. Then
	$$
	G_h^w(a\sharp b)=G_h^w(a)\circ G_h^w(b),
	$$
	where
	\begin{equation}
		a\sharp b(x,\xi):=\frac{1}{(\pi h)^{2}}\int_\mathbb{R}\int_\mathbb{R}\int_\mathbb{R}\int_\mathbb{R}
		e^{\frac{2i}{h}(\eta z- y\zeta)}a(x+z,\xi+\zeta)b(x+y,\xi+\eta)dyd\eta dz d\zeta. \notag
	\end{equation}
\end{prop}

\begin{prop}[Proposition 2.4 in \cite{Zhang}]\label{P2.3-0}
	Suppose that $a,b\in\mathcal{S}(\mathbb{R}^2)$. Then
	\begin{equation}
		a\sharp b=ab+\frac{ih}{2} (\partial_x a\partial_\xi b-\partial_\xi a\partial_x b)+R, \notag
	\end{equation}
	where 
	\begin{eqnarray}
		R&=&\frac{1}{(2\pi  )^{2}}\int_\mathbb{R}\int_\mathbb{R}\int_\mathbb{R}\int_\mathbb{R}
		e^{\frac{2i}{h}(\eta z- y\zeta)}
		\left\{-
		\int^1_0\partial_x^2a(x+tz,\xi)(1-t)dt
		\partial_\eta^2 b (x+y,\xi+\eta)\right.\nonumber\\
		&&\left.+\int^1_0\int^1_0\partial_x\partial_\xi a(x+sz,\xi+t\zeta)  dsdt\partial_\eta\partial_y b (x+y,\xi+\eta)
		\right.\nonumber\\
		&&\left. -\int^1_0\partial_\xi^2a(x,\xi+t\zeta)(1-t) dt\partial_y^2 b (x+y,\xi+\eta)\right\}
		dyd\eta dz d\zeta.\notag
	\end{eqnarray}
\end{prop}

\begin{lem}\label{l0}
	Assume that $\Gamma_{0}(\xi)$ is a smooth function, satisfying $|\partial^\alpha \Gamma_{0}(\xi)|
	\le C_\alpha  <\xi>^{-1-\alpha}$ for any $\alpha \in \mathbb{N}$. Then we have 
	\begin{equation}\label{4131}
		\Gamma_{0}(\frac{x+F'(\xi)}{\sqrt{h}})\sharp\frac{x+F'(\xi)}{\sqrt{h}}=	\frac{x+F'(\xi)}{\sqrt{h}}\sharp \Gamma_{0}(\frac{x+F'(\xi)}{\sqrt{h}})=\Gamma_{0}(\frac{x+F'(\xi)}{\sqrt{h}})\frac{x+F'(\xi)}{\sqrt{h}}.
	\end{equation}
	In addition, if  $|\partial^\alpha (\xi\Gamma_0(\xi))|
	\le C_\alpha  <\xi>^{-1-\alpha}$ for any $\alpha \in \mathbb{N}$, then we have
	\begin{equation}\label{4132}
		\left(	\Gamma_{0}(\frac{x+F'(\xi)}{\sqrt{h}})\sharp\frac{x+F'(\xi)}{\sqrt{h}}\right)\sharp \frac{x+F'(\xi)}{\sqrt{h}}=\Gamma_{0}(\frac{x+F'(\xi)}{\sqrt{h}})(\frac{x+F'(\xi)}{\sqrt{h}})^2.
	\end{equation}
\end{lem}
\begin{proof}
	An application of Proposition \ref{P2.3-0} yields 
	\begin{eqnarray}
		&&	\Gamma_{0}(\frac{x+F'(\xi)}{\sqrt{h}})\sharp \frac{x+F'(\xi)}{\sqrt{h}}\notag\\
		&=& \Gamma_{0}(\frac{x+F'(\xi)}{\sqrt{h}})  \frac{x+F'(\xi)}{\sqrt{h}}+\frac{ih}{2}\left(\partial_x\Gamma_{0}(\frac{x+F'(\xi)}{\sqrt{h}})\partial_\xi (\frac{x+F'(\xi)}{\sqrt{h}})\right. \notag\\
		&&\left.\qquad-\partial_\xi\Gamma_{0}(\frac{x+F'(\xi)}{\sqrt{h}})\partial_x (\frac{x+F'(\xi)}{\sqrt{h}})\right)+R\notag\\
		&=& \Gamma_{0}(\frac{x+F'(\xi)}{\sqrt{h}}) (\frac{x+F'(\xi)}{\sqrt{h}}),\notag
	\end{eqnarray}
	where  $R=0$, since $\partial_{\xi\xi }(\frac{x+F'(\xi)}{\sqrt{h}})=\partial_{xx} (\frac{x+F'(\xi)}{\sqrt{h}})=\partial_{x\xi} (\frac{x+F'(\xi)}{\sqrt{h}})=0$. 
	Similarly, we can show that 
	\begin{eqnarray}
		\frac{x+F'(\xi)}{\sqrt{h}}\sharp \Gamma_{0}(\frac{x+F'(\xi)}{\sqrt{h}})=\Gamma_{0}(\frac{x+F'(\xi)}{\sqrt{h}}) (\frac{x+F'(\xi)}{\sqrt{h}}).\notag
	\end{eqnarray}This completes the proof of  (\ref{4131}).  Finally, (\ref{4132}) follows easily from (\ref{4131}): 
	\begin{eqnarray}
		\Gamma_0(\frac{x+F'(\xi)}{\sqrt h})(\frac{x+F'(\xi)}{\sqrt h})^2&=&\left(\Gamma_{0}(\frac{x+F'(\xi)}{\sqrt h})\frac{x+F'(\xi)}{\sqrt h}\right)\sharp \frac{x+F'(\xi)}{\sqrt h}\notag\\
		&=&\left(\Gamma_{0}(\frac{x+F'(\xi)}{\sqrt h})\sharp \frac{x+F'(\xi)}{\sqrt h}\right)\sharp \frac{x+F'(\xi)}{\sqrt h}.\notag
	\end{eqnarray}
\end{proof}

\section{Proof of Theorem \ref{T1.1}}\label{S3}
In this section, we prove  the   global existence and uniqueness results in Theorem \ref{T1.1}. 
We will use the following  lemmas. 
\begin{lem}\label{ll1}
	Assume $f:[1,T]\times \mathbb{R}\rightarrow \mathbb{C},\ T>1$ is a smooth function, there exists a positive constant $C$ independent of $T, f$ such that for all $t\in[1,T]$
	\begin{equation}
		\left\|f(t,x)\right\|_{L^\infty _x}\le Ct^{-1/2}\left\|f(t,x)\right\|_{L^2 _x}^{1/2}\left\|\mathcal{L}f(t,x)\right\|_{L^2 _x}^{1/2},\label{4133}
	\end{equation}
	\begin{equation}
		\|\mathcal{L} f(t,x)\|_{L^4_x}\le C\|f(t,x)\|_{L^\infty _x}^{1/2}\|\mathcal{L} ^2f(t,x)\|_{L^2_x}^{1/2}.\label{4134}
	\end{equation}
\end{lem}
\begin{proof}
	To begin with, it is useful to introduce a certain phase function. Let 
	\begin{equation}
		\phi(t,x)=\frac{x^2+2c_1tx}{4c_2t}.\notag
	\end{equation}
	Since 
	\begin{equation}
		\mathcal{L}=x+tF'(D)=x+c_1t-2c_2it\partial_x,\label{1234}
	\end{equation}
	it is straightforward to check that 
	\begin{equation}
		-2c_2it \partial_x (f(t,x)e^{i\phi(t,x)})=e^{i\phi(t,x)}\mathcal{L} f(t,x), \label{4135}
	\end{equation}
	\begin{equation}
		-4c_2^2t^2 \partial_{xx} (f(t,x)e^{i\phi(t,x)})=e^{i\phi(t,x)}\mathcal{L}^2 f(t,x).\label{4136}
	\end{equation}
	From (\ref{4135}), we have, for $t\in [1,T]$,  
	\begin{eqnarray}\label{45200}
		\left\|\partial_x (f(t,x)e^{i\phi(t,x)})\right\|_{L_x^2}\le \frac{C}{t}\left\|\mathcal{L}f(t,x)\right\|_{L_x^2}.
	\end{eqnarray}
	On the other hand, using Gagliardo-Nirenberg's inequality, we obtain 
	\begin{eqnarray}
		\left\|f(t,x)\right\|_{L^\infty _x}=\left\|f(t,x)e^{i\phi(t,x)}\right\|_{L^\infty _x}\le C\left\|f(t,x)e^{i\phi(t,x)}\right\|_{L_x^2}^{1/2}\left\|\partial_x (f(t,x)e^{i\phi(t,x)})\right\|_{L_x^2}^{1/2}.\notag
	\end{eqnarray}
	This  together with (\ref{45200}) yields (\ref{4133}). 
	
	Similarly, it follows from  (\ref{4135}) and Gagliardo-Nirenberg's inequality that 
	\begin{eqnarray}
		\|\mathcal{L} f(t,x)\|_{L_x^4}\le  Ct \|f(t,x)e^{i\phi(t,x)}\|_{L^\infty }^{1/2}\|\partial_{xx}(f(t,x)e^{i\phi(t,x)})\|_{L^2}^{1/2},\notag
	\end{eqnarray}
	which together with  (\ref{4136}) gives the desired estimate  (\ref{4134}). 
\end{proof}

Using the classical energy estimate method, we obtain the following lemma easily and omit the details.
\begin{lem}\label{l4.1}
	Assume  $\text{Im} \lambda >0$, and $u\in C([0,T]; L^2)$ is a  solution of (\ref{1.1}), then we have
	\begin{equation}
		\|u(t,\cdot)\|_{L^2}\le\|u_0\|_{L^2}, \ t\in[0,T].\notag
	\end{equation}
\end{lem}
\begin{proof}[\textbf{Proof of Theorem \ref{T1.1}}]
	For the initial datum $u_0\in H^{0,1}$, the  existence and uniqueness  of a  local strong $L^2$  solution  to the Cauchy problem (\ref{1.1}) easily follow from  Strichartz's estimate 
	\begin{eqnarray}
		\left\|u\right\|_{L^\infty L^2\cap L^4 L^\infty }\le C\left\|u_0\right\|_{L^2}+C\left\||u|^\alpha u\right\|_{L^1L^2}, \notag
	\end{eqnarray}
	and a standard contraction argument. Moreover,  by applying Lemma \ref{l4.1},  we can extend this local solution to  $[0,\infty  ) $ easily and  omit the details. 
	
	In what follows, we prove the estimate (\ref{3192})--(\ref{491}) and thus completing the proof of Theorem \ref{T1.1}.
	
	\noindent \textbf{Estimate of $\left\|\mathcal{L}u\right\|_{L^2}$.} 
	Applying the operator ${\mathcal{L}}$ to  (\ref{1.1}),  and using the fundamental commutation property $
	[D_t-F(D),\mathcal{L}]=0$, we get 
	\begin{equation}
		(D_t-F(D))\mathcal{L}u=\lambda \mathcal{L}(\left|u\right|^\alpha u).\notag
	\end{equation}
	This implies 
	\begin{equation}
		\frac{1}{2}	\frac{d}{dt}\|{\mathcal{L}} u\|_{L^2}^2=-\text{Im} \int_{\mathbb{R}}  \lambda  {\mathcal{L}}(\left|u\right|^\alpha u)\overline{{\mathcal{L}}u} \mathrm{d}x.\notag
	\end{equation}
	On the other hand, from the expressions of  $\mathcal{L} $ in  (\ref{1234}),  it is straightforward to check that 
	\begin{eqnarray}\label{z1}
		\mathcal{L}(\left|u\right|^\alpha u)=\frac{\alpha +2}{2}\left|u\right|^{\alpha }	\mathcal{L}u-\frac{\alpha }{2}\left|u\right|^{\alpha -2}u^2\overline{	\mathcal{L} u};
	\end{eqnarray}
	and that for $\alpha \ge1$,
	\begin{eqnarray}\label{z2}
		\mathcal{L}^2(|u|^\alpha u)&=&\frac{\alpha +2}{2}|u|^\alpha \mathcal{L}^2u+\frac{\alpha }{2}|u|^{\alpha -2}u^2\overline{\mathcal{L}^2u}+\frac{(\alpha +2)\alpha }{2}|u|^{\alpha -2}\text{Im} (\mathcal{L}u \overline{u})\mathcal{L}u \notag\\
		&&-\frac{\alpha (\alpha -2)}{2}|u|^{\alpha -4}u^2 \text{Im} (\mathcal{L}u \overline{u})\overline{\mathcal{L}u}-\alpha |u|^{\alpha -2}u|\mathcal{L}u|^2.
	\end{eqnarray}
	From (\ref{z1}) and the large dissipative condition (\ref{lambda1}), we have   for all $t\ge0$, 
	\begin{eqnarray}
		&&-\text{Im} \left(\lambda {\mathcal{L}}(\left|u\right|^\alpha u)\overline{{\mathcal{L}}u}\right)\notag\\
		&=&-\text{Im} \lambda \frac{\alpha +2}{2}\left|u\right|^\alpha \left| {\mathcal{L} }u\right|^2+\text{Im} (\lambda \frac{\alpha }{2}\left|u\right|^{\alpha -2}u^2\overline{{\mathcal{L} }u}^2)\notag\\
		&\le& (-\frac{\alpha +2}{2}\lambda _2+\frac{\alpha }{2}|\lambda |)\left|u\right|^\alpha |{\mathcal{L} }u|^2\le0.\label{151}
	\end{eqnarray}
	Therefore, for all $t\ge0$, we have 
	\begin{equation}\label{3234}
		\|{\mathcal{L}}u(t)\|_{L^2}\le \left\|{{\mathcal L}}u(0)\right\|_{L^2}=\left\|xu_0\right\|_{L^2},
	\end{equation}
	which together with Lemma \ref{l4.1} yields (\ref{3192}). 
	
\noindent\textbf{Estimate of $\left\|\mathcal{L}^2u\right\|_{L^2}$. }  Using the commutation relation 	$[D_t-F(D),\mathcal{L}^2]=0$, we get 
	\begin{eqnarray}
		\frac{1}{2}\frac{d}{dt}\left\|\mathcal{L}^2u\right\|_{L^2}^2=-\text{Im}  \int_{\mathbb{R}}\lambda \mathcal{L}^2(|u|^\alpha u)\cdot \overline{\mathcal{L}^2u} \mathrm{d}x.\label{432}
	\end{eqnarray}
	From (\ref{z2}) and  the large dissipative condition  (\ref{lambda1}), we have that  for all $t\ge0$,  
	\begin{eqnarray}
		&&-\text{Im}  \int_{\mathbb{R}}\lambda \mathcal{L}^2(|u|^\alpha u)\cdot \overline{\mathcal{L}^2u} \mathrm{d}x\notag\\
		&\le& -\lambda _2\frac{\alpha +2}{2}\int_{\mathbb{R}}|u|^\alpha |\mathcal{L}^2u|^2 \mathrm{d}x+\frac{\alpha }{2}|\lambda | \lambda \int_{\mathbb{R}}|u|^{\alpha }|\mathcal{L} ^2 u|^2 \mathrm{d}x\notag\\
		&&+C\int_{\mathbb{R}}|u|^{\alpha -1}|\mathcal{L}u|^2|\mathcal{L}^2u| \mathrm{d}x\notag\\
		&\le&C\left\|u\right\|_{L^\infty }^{\alpha -1}\left\|\mathcal{L}u\right\|_{L^\infty }\left\|\mathcal{L}u\right\|_{L^2}\left\|\mathcal{L}^2u\right\|_{L^2}.\notag
	\end{eqnarray}
	An application of  Lemma \ref{ll1} then yields 
	\begin{eqnarray}
		&&-\text{Im}  \int_{\mathbb{R}}\lambda\mathcal{L}^2(|u|^\alpha u)\cdot \overline{\mathcal{L}^2u} \mathrm{d}x\notag\\
		&\le& Ct^{-\alpha /2}\left\|u\right\|_{L^2}^{\frac{\alpha -1}{2}}\left\|\mathcal{L}u\right\|_{L^2}^{\frac{\alpha +2}{2}}\left\|\mathcal{L}^2u\right\|_{L^2}^{\frac{3}{2}}\notag\\
		&\le& Ct^{-\alpha /2}\left\|u_0\right\|_{H^{0,1}}^{\alpha +\frac{1}{2}}\left\|\mathcal{L}^2u\right\|_{L^2}^{\frac{3}{2}},\label{433}
	\end{eqnarray}
	where the last inequality holds by applying (\ref{3234}) and Lemma \ref{l4.1}. Substituting (\ref{433}) into  (\ref{432}), we get   
	\begin{eqnarray}
		\frac{d}{dt}\left\|\mathcal{L}^2u\right\|_{L^2}^{1/2}\le Ct^{-\alpha /2}\left\|u_0\right\|_{H^{0,1}}^{\alpha +1/2}.\notag
	\end{eqnarray}
	Hence
	\begin{eqnarray}
		\left\|\mathcal{L}^2u\right\|_{L^2}\le \left\|u_0\right\|_{H^{0,2}}+Ct^{2-\alpha }\left\|u_0\right\|_{H^{0,1}}^{2\alpha +1},\label{0451}
	\end{eqnarray}
	which yields (\ref{491}). 
	
	Finally,  from Lemma \ref{ll1}, Lemma \ref{l4.1} and (\ref{3234}),  the desired decay estimate (\ref{ldecay}) follows. 
\end{proof}

\section{The proof of  Theorem \ref{T1.2}}\label{S5}
This section is devoted to proving   Theorem \ref{T1.2}.  It is organized in two subsections. 

In the first one, we make first a  semiclassical change of variables (\ref{4.2}) and then prove in Lemma \ref{l4.3-00}  that  $v_{\Lambda ^c}$ can be considered as a remainder. In addition,  we derive  the ODE (\ref{4.34}) for  $v_{\Lambda }$ and  then estimate  the remainders  $R_1(v),\ R_2(v)$ in the rest of this subsection.

 In the second one, we use the ODE (\ref{4.34}) to derive the uniform for  $v_{\Lambda }$ and then in the solution  $u$, combining the bootstrap and the  contradiction argument.  
\subsection{Semiclassical reduction of the problem}\label{sub1}
We rewrite the problem in the semiclassical framework. Set
\begin{equation}\label{4.2}
	u(t,x)=\frac{1}{\sqrt{t}}v(t,\frac{x}{t}),\ h=\frac{1}{t}.
\end{equation}
Then the equation (\ref{1.1}) is rewritten  as
\begin{equation}
	(D_t-G_h^w(x\xi+F(\xi)))v=\lambda t^{-\alpha /2}|v|^\alpha v.\label{3.4-000}
\end{equation}
At the same time, set
\begin{equation}
	\widetilde{\mathcal{L}}=\frac{1}{h}G_h^w(x+F'(\xi)).\label{3.4}
\end{equation}
Indeed,  since $F'(\xi)=2c_2\xi+c_1$, we have that  $	\widetilde{\mathcal{L}}=(x+c_1)t+2c_2D_x$ so  that  
\begin{equation}\label{4103}
	\mathcal{L} u(t,x)=\frac{1}{\sqrt t}(\widetilde{\mathcal{L}}v)(t,\frac{x}{t}),\qquad \mathcal{L} ^2u(t,x)=\frac{1}{ \sqrt {t}}(\widetilde{\mathcal{L} }^2v)(t,\frac{x}{t}).
\end{equation}
Moreover, one has   
\begin{equation}
	\|u(t,\cdot)\|_{L^2}=\|v(t,\cdot)\|_{L^2},\  
	\  \|u(t,\cdot)\|_{L^\infty}=t^{-1/2}\|v(t,\cdot)\|_{L^\infty}, \label{3.4-0}
\end{equation}
and 
\begin{equation}\label{4104}
	\|\mathcal{L}u(t,\cdot)\|_{L^2}=\|\widetilde{\mathcal{L}}v(t,\cdot)\|_{L^2}, \ \left\|\mathcal{L} ^2u(t,\cdot)\right\|_{L^2}=\|\widetilde{\mathcal{L} }^2v(t,\cdot)\|_{L^2}.
\end{equation}
Therefore the proof of Theorem \ref{T1.2} reduces to estimate $\left\|v(t,\cdot)\right\|_{L^\infty }$. To do so,  we decompose  $v=v_\Lambda+v_{\Lambda^c}$ with
\begin{equation}\label{3194}
	v_\Lambda=G^w_h(\Gamma) v,
\end{equation}
where $\Gamma(x,\xi)=\gamma(\frac{x+F'(\xi)}{\sqrt{h}})$ with $\gamma\in C^\infty_0(\mathbb{R})$ satisfying that $\gamma\equiv1$ in a neighbourhood of zero.

We have the following $L^\infty-$estimates for $v_{\Lambda^c}$, which shows that  $v_{\Lambda }$ can be considered as a small perturbation of  $v$. 
\begin{lem}\label{l4.3-00}
	Assume  $u_0\in H^{0,1}$.  Then for all $t\ge1$,
	\begin{equation}
		\|v_{\Lambda^c}(t,\cdot)\|_{L^\infty}\leq C\|u_0\|_{H^{0,1}}t^{-1/4} ,\ \|v_{\Lambda^c}(t,\cdot)\|_{L^2}\leq  C\|u_0\|_{H^{0,1}}t^{-1/2}. \notag
	\end{equation}
\end{lem}
\begin{proof}
	Let $\Gamma_{-1}(\xi)=\frac{1-\gamma(\xi)}{\xi},$   satisfying $|\partial^\alpha \Gamma_{-1}(\xi)|
	\le C_\alpha <\xi>^{-1-\alpha}$ for any $\alpha\in \mathbb{N}$.  Then we can write  
	\begin{eqnarray} 
		1-\Gamma(x,\xi)
		&=&\sqrt h\Gamma_{-1}(\frac{x+F'(\xi)}{\sqrt{h}})(\frac{x+F'(\xi)}{{h}}). \notag
	\end{eqnarray}	 
	From Lemma \ref{l0}, and the definition of $\widetilde{\mathcal{L}}$ in (\ref{3.4}), we get 
	\begin{eqnarray}
		&&G_h^w(1-\Gamma)v=\sqrt h G_h^w(\Gamma_{-1}(\frac{x+F'(\xi)}{\sqrt h}))\circ (\widetilde{\mathcal{L}}v).\notag
	\end{eqnarray}
	An application of  Proposition \ref{P2.6} yields,  for all $t\ge1$,  
	\begin{equation}
		\label{b1}
		\left\|G_h^w(1-\Gamma)v\right\|_{L^\infty }\le C\|\widetilde{\mathcal{L}}v\|_{L^2}h^{1/4},\ \left\|G_h^w(1-\Gamma)v\right\|_{L^2 }\le C\|\widetilde{\mathcal{L}}v\|_{L^2}h^{1/2}.
	\end{equation} 
	Since  $\|\widetilde{\mathcal{L} }v\|_{L^2}\le C\|u_0\|_{H^{0,1}}$ by (\ref{4104}) and (\ref{3192}),  we obtain the desired estimate in   Lemma \ref{l4.3-00}. 
\end{proof}
To get the $L^\infty -$ estimates for $v_\Lambda $ in large time, we deduce an ODE from the PDE system (\ref{3.4-000}):
\begin{equation}
	D_tv_\Lambda=w(x)v_{\Lambda }+\lambda t^{-\alpha /2}  |v_\Lambda|^\alpha v_\Lambda+ t^{-\alpha /2} \left(R_1(v)+R_2(v)\right)\label{4.34}
\end{equation}
where  $w(x)=-(x+c_1)^2/(4c_2)+c_0$ and 
\begin{eqnarray}
	R_1(v)&=&t^{\alpha /2}[D_t-G^w_h(x\xi+F(\xi)),G^w_h(\Gamma)]v\notag\\
	&&+t^{\alpha /2}\left(G_h^w(x\xi+F(\xi))-w(x)\right)v_{\Lambda},\notag
\end{eqnarray}
\begin{equation}
	R_2(v)=-\lambda G^w_h(1-\Gamma)(|v|^\alpha v)+\lambda \left( |v|^\alpha v-|v_\Lambda|^\alpha v_\Lambda\right). \notag
\end{equation}
The rest of this subsection is devoted to proving the following estimates for the remainders  $R_1(v)$ and  $R_2(v)$.
\begin{prop}
	\label{lp}
	For any  $t\ge1$, we have \\
	(a) when $u_0\in H^{0,1}$: 
	\begin{equation}
		\|R_1(v)\|_{L^\infty}\leq C \left\|u_0\right\| _{H^{0,1}}t^{-5/4+\alpha /2},\notag
	\end{equation}
	\begin{equation}
		\|R_1(v)\|_{L^2}\leq C\left\|u_0\right\| _{ H^{0,1}}t^{-3/2+\alpha /2},\notag
	\end{equation}
	\begin{equation}
		\|R_2(v)\|_{L^\infty }\leq C \left\|u_0\right\| _{H^{0,1}}(\left\|v\right\|_{L^\infty }^\alpha+\left\|v_{\Lambda }\right\|_{L^\infty }^\alpha ) t^{-1/4},\notag
	\end{equation}
	\begin{equation}
		\|R_2(v)\|_{L^2}\leq C \left\|u_0\right\| _{H^{0,1}}(\left\|v\right\|_{L^\infty }^\alpha+\left\|v_{\Lambda }\right\|_{L^\infty }^\alpha ) t^{-1/2}.\notag
	\end{equation}
	(b) when $u_0\in H^{0,2}$: 
	\begin{equation}
		\|R_1(v)\|_{L^\infty}\leq C (\|u_0\|_{H^{0,2}}+\|u_0\|_{H^{0,2}}^{2\alpha +1})t^{1/4-\alpha/2 },\notag
	\end{equation}
	\begin{equation}
		\|R_1(v)\|_{L^2}\leq C(\|u_0\|_{H^{0,2}}+\|u_0\|_{H^{0,2}}^{2\alpha +1})t^{-\alpha/2 },\notag
	\end{equation}
	\begin{equation}\notag
		\|R_2(v)\|_{L^\infty }\leq C (\|u_0\|_{H^{0,2}}+\|u_0\|_{H^{0,2}}^{2\alpha +1})(\left\|v\right\|_{L^\infty }^\alpha+\left\|v_{\Lambda }\right\|_{L^\infty }^\alpha ) t^{5/4-\alpha },
	\end{equation}
	\begin{equation}\notag
		\|R_2(v)\|_{L^2}\leq C (\|u_0\|_{H^{0,2}}+\|u_0\|_{H^{0,2}}^{2\alpha +1})(\left\|v\right\|_{L^\infty }^\alpha+\left\|v_{\Lambda }\right\|_{L^\infty }^\alpha ) t^{1-\alpha }.
	\end{equation}
\end{prop}
The proof of  Proposition \ref{lp} is split into  Lemmas \ref{l4.3}--\ref{l5.6} below. 
\begin{lem}\label{l4.3}
	If  $u_0\in H^{0,1}$, then  for all $t\ge1$,
	\begin{eqnarray}
		\|[D_t-G^w_h(x\xi+F(\xi)),G^w_h(\Gamma)]v\|_{L^\infty}\leq C\|\widetilde{\mathcal{L} }v\|_{L^2}t^{-5/4},\notag\\ \|[D_t-G^w_h(x\xi+F(\xi)),G^w_h(\Gamma)]v\|_{L^2} \le C \|\widetilde{\mathcal{L} }v\|_{L^2}t^{-3/2}.\notag
	\end{eqnarray}
	Furthermore, if $u_0\in H^{0,2}$ and $\alpha \ge1$, then for all $t\ge1$, 
	\begin{eqnarray}
		\|[D_t-G^w_h(x\xi+F(\xi)),G^w_h(\Gamma)]v\|_{L^\infty}\leq C\|\widetilde{\mathcal{L} }^2v\|_{L^2}t^{-7/4},\notag\\  	\|[D_t-G^w_h(x\xi+F(\xi)),G^w_h(\Gamma)]v\|_{L^2} \le C\|\widetilde{\mathcal{L} }^2v\|_{L^2}t^{-2}. \notag
	\end{eqnarray}
\end{lem}
\begin{proof}
	Since $h=t^{-1}$, by a direct computation, we have
	\begin{eqnarray}\label{3291}
		&&	[D_t,G^w_h(\Gamma)]f\notag\\
		&=&-hi G^w_h(\Gamma)f+
		\frac{1}{2\pi h}\int_{\mathbb{R}}\int_{\mathbb{R}} e^{i(x-y)\frac{\xi}{h}}(x-y)\xi \gamma(\frac{\frac{x+y}{2}+F'(\xi)}{\sqrt h})f(t,y)dyd\xi\nonumber\\
		&&+\frac{1}{2\pi h}\int_{\mathbb{R}}\int_{\mathbb{R}} e^{i(x-y)\frac{\xi}{h}} \gamma'(\frac{\frac{x+y}{2}+F'(\xi)}{\sqrt{h}})
		(\frac{x+y}{2}+F'(\xi))\frac{\sqrt{h}}{2i}f(t,y)dyd\xi\nonumber\\
		&=& ih G^w_h[\gamma'(\frac{x+F'(\xi)}{\sqrt{h}})(\xi\frac{F''(\xi)}{\sqrt{h}}-\frac{x+F'(\xi)}{2\sqrt{h}})]f, \label{4.12}
	\end{eqnarray}
	where we used the fact that
	\begin{eqnarray}
		&& \frac{1}{2\pi h}\int_{\mathbb{R}}\int_{\mathbb{R}} e^{i(x-y)\frac{\xi}{h}}(x-y)\xi \gamma(\frac{\frac{x+y}{2}+F'(\xi)}{\sqrt h})f(t,y)dyd\xi\nonumber\\
		&=&
		\frac{1}{2\pi h}\int_{\mathbb{R}}\int_{\mathbb{R}}\frac{h}{i} \partial_\xi e^{i(x-y)\frac{\xi}{h}} \xi \gamma(\frac{\frac{x+y}{2}+F'(\xi)}{\sqrt h})f(t,y)dyd\xi\nonumber\\
		&=& hi G^w_h(\Gamma)f+\frac{1}{2\pi h}\int_{\mathbb{R}}\int_{\mathbb{R}} e^{i(x-y)\frac{\xi}{h}}\xi \gamma'(\frac{\frac{x+y}{2}+F'(\xi)}{\sqrt{h}})
		F''(\xi)i\sqrt{h}f(t,y)dyd\xi.\notag
	\end{eqnarray}
	Then using  Propositions \ref{ab} and \ref{P2.3-0} we write 
	\begin{equation}
		[G^w_h(x\xi+F(\xi)),G^w_h(\Gamma)]=ih G^w_h(\gamma'(\frac{x+F'(\xi)}{\sqrt{h}})(\xi\frac{F''(\xi)}{\sqrt{h}}-\frac{x+F'(\xi)}{\sqrt{h}}))
		+r_1-r_2,\label{4.14}
	\end{equation}
	where
	\begin{eqnarray}
		r_1
		&=&\frac{1}{(2\pi)^2}\int_\mathbb{R}\int_\mathbb{R}\int_\mathbb{R}\int_\mathbb{R}
		e^{\frac{2i}{h} (\eta z-y \zeta)}\frac{1}{h}\gamma''(\frac{x+y+F'(\xi+\eta)}{\sqrt h})\left\{F''(\xi+\eta)\phantom{\int_0^1}\right. \notag\\
		&&\left.-\int _{0}^1F''(\xi+t\zeta)(1-t)dt\right\}
		dyd\eta dz d\zeta,\notag
	\end{eqnarray}
	\begin{eqnarray}
		r_2
		&=&\frac{1}{(2\pi)^2}\int_\mathbb{R}\int_\mathbb{R}\int_\mathbb{R}\int_\mathbb{R}
		e^{\frac{2i}{h} (\eta z-y \zeta)} \left\{-\frac{1}{h}\int_0^1 \gamma''(\frac{x+tz+F'(\xi)}{\sqrt h})(1-t)dtF''(\xi+\eta)\right. \notag\\
		&&\left.+\int_0^1\int_0^1 \frac{1}{h}\gamma''(\frac{x+sz+F'(\xi+t\zeta)}{\sqrt h})F''(\xi+t\zeta)dsdt\right\}
		dyd\eta dz d\zeta.\notag
	\end{eqnarray}
	Since   $F''=2c_2$ and  $
	\int_{ \mathbb{R}}e^{\frac{2i\eta z}{h}}dz=\delta(\eta)\pi h$, $\int_{ \mathbb{R}}e^{-\frac{2iy\zeta}{h}}d\zeta=\delta(y)\pi h$, we have $r_1=\frac{c_2h}{4}\gamma''(\frac{x+F'(\xi)}{\sqrt h})$. Similarly, we have  $r_2=\frac{c_2h}{4}\gamma''(\frac{x+F'(\xi)}{\sqrt h})$. So we deduce from   (\ref{3291}) and (\ref{4.14}) that 
	\begin{eqnarray}\label{451}
		[D_t-G^w_h(x\xi+F(\xi)),G^w_h(\Gamma)]v=\frac{ih}{2}G_h^w(\gamma'(\frac{x+F'(\xi)}{\sqrt h})\frac{x+F'(\xi)}{\sqrt{h}})v. 
	\end{eqnarray}
	On the other hand, using  Lemma \ref{l0},  we can rewrite (\ref{451}) as 
	\begin{eqnarray}
		[D_t-G^w_h(x\xi+F(\xi)),G^w_h(\Gamma)]v=\frac{ih^{3/2}}{2} G_h^w(\gamma'(\frac{x+F'(\xi)}{\sqrt h}))\circ (\widetilde{\mathcal{L}}v),\ \text{when } u_0\in H^{0,1}\notag
	\end{eqnarray}
	\begin{eqnarray}
		[D_t-G^w_h(x\xi+F(\xi)),G^w_h(\Gamma)]v	=\frac{ih^2}{2} G_h^w(\Gamma_{-2}(\frac{x+F'(\xi)}{\sqrt h}))\circ (\widetilde{\mathcal{L}}^2v),\ \text{when } u_0\in H^{0,2}\notag
	\end{eqnarray}
	where $\Gamma_{-2}(\xi)=\frac{\gamma'(\xi)}{\xi}$ satisfying $|\partial^\alpha \Gamma_{-2}(\xi)|
	\le C_\alpha <\xi>^{-1-\alpha}$ for any $\alpha\in \mathbb{N}$. 
	The above identities and an application of 	 Proposition \ref{P2.6} yields the desired estimates in  Lemma \ref{l4.3}.   
\end{proof}
\begin{lem}\label{l5}
	If  $u_0\in H^{0,1}$, then  for all $t\ge1$,
	\begin{equation}
		\|\left(G_h^w(x\xi+F(\xi))-w(x)\right)v_{\Lambda}\|_{L^\infty}
		\leq C\|\widetilde{\mathcal{L} }v\|_{L^2}t^{-5/4},\notag
	\end{equation}
	\begin{equation}
		\|\left(G_h^w(x\xi+F(\xi))-w(x)\right)v_{\Lambda}\|_{L^2}\le C\|\widetilde{\mathcal{L} }v\|_{L^2}t^{-3/2}.\notag
	\end{equation}
	Furthermore, if $u_0\in H^{0,2}$ and $\alpha \ge1$, then for all $t\ge1$, 
	\begin{equation}
		\|\left(G_h^w(x\xi+F(\xi))-w(x)\right)v_{\Lambda}\|_{L^\infty}
		\leq C\|\widetilde{\mathcal{L} }^2v\|_{L^2}t^{-7/4}\notag
	\end{equation}
	\begin{equation}
		\|\left(G_h^w(x\xi+F(\xi))-w(x)\right)v_{\Lambda}\|_{L^2}\le C\|\widetilde{\mathcal{L} }^2v\|_{L^2}t^{-2}.\notag
	\end{equation}
\end{lem}
\begin{proof}
Applying  Lemma \ref{l0} to (\ref{1510}) we get
	\begin{eqnarray}
		&&\left(G_h^w(x\xi+F(\xi))-w(x)\right)v_{\Lambda}=\frac{1}{4c_2}G_h^w((x+F'(\xi))^2\gamma (\frac{x+F'(\xi)}{\sqrt h}))v\notag\\
		&&= 	\frac{1}{4c_2}\begin{cases}
			h^{3/2}G_h^w(\Gamma_{-3}(\frac{x+F'(\xi)}{\sqrt{h}})(\widetilde{\mathcal{L}}v),\ u_0\in H^{0,1} \\
			h^2G_h^w(\gamma(\frac{x+F'(\xi)}{\sqrt{h}})(\widetilde{\mathcal{L}}^2v),\ u_0\in H^{0,2}
		\end{cases}\notag
	\end{eqnarray}
	where $\Gamma_{-3}(\xi)=\xi \gamma(\xi)$ satisfying  $|\partial^\alpha \Gamma_{-3}(\xi)|
	\le C_\alpha <\xi>^{-1-\alpha}$ for any $\alpha\in \mathbb{N}$.  
	The above identities and an application of 	 Proposition \ref{P2.6} yields the desired estimates in  Lemma \ref{l5}.   
\end{proof}
\begin{lem}\label{l4.5}
	If  $u_0\in H^{0,1}$, then  for all $t\ge1$,
	\begin{equation}
		\|G^w_h(1-\Gamma)(|v|^\alpha v)\|_{L^\infty}\le C\|v\|_{L^\infty }^\alpha \|\widetilde{\mathcal{L} }v\|_{L^2}t^{-1/4},\notag
	\end{equation}
	\begin{equation}
		\|G^w_h(1-\Gamma)(|v|^\alpha v)\|_{L^2}\le C\|v\|_{L^\infty }^\alpha \|\widetilde{\mathcal{L}}v\|_{L^2}t^{-1/2}.\notag
	\end{equation}
	Furthermore, if $u_0\in H^{0,2}$ and $\alpha \ge1$, then  for all $t\ge1$, 
	\begin{equation}
		\|G^w_h(1-\Gamma)(|v|^\alpha v)\|_{L^\infty}\le C\|v\|_{L^\infty }^\alpha  \|\widetilde{\mathcal{L} }^2v\|_{L^2}t^{-3/4}, \notag
	\end{equation}
	\begin{equation}
		\|G^w_h(1-\Gamma)(|v|^\alpha v)\|_{L^2}\le C\|v\|_{L^\infty }^\alpha\|\widetilde{\mathcal{L} }^2v\|_{L^2}t^{-1}. \notag
	\end{equation}
\end{lem}
\begin{proof}
	We first claim that 
	\begin{equation}
		\|\widetilde{\mathcal{L} }(|v|^\alpha v)\|_{L^2}\le C\|v\|_{L^\infty }^\alpha \|\widetilde{\mathcal{L} }v\|_{L^2},\label{asd1}
	\end{equation}
	and for  $\alpha \ge1$,
	\begin{equation}
		\|\widetilde{\mathcal{L} }^2(|v|^\alpha v)\|_{L^2}\le C\|v\|_{L^\infty }^\alpha \|\widetilde{\mathcal{L} }^2v\|_{L^2}.\label{asd2}
	\end{equation}
	We only give the proof of (\ref{asd2}) as   (\ref{asd1}) can be proved in a similar way. Since 
	\begin{equation}
		\mathcal{L} ^2(|u|^\alpha u)=t^{-\frac{\alpha +1}{2}}\widetilde{\mathcal{L} }^2(|v|^\alpha v)(t,\frac{x}{t})\notag
	\end{equation}
	by the second formula in  (\ref{4103}), it follows from   Lemma \ref{ll1} and  (\ref{z2}) that 
	\begin{eqnarray}
		\|\widetilde{\mathcal{L} }^2(|v|^\alpha v)\|_{L^2}&\lesssim &  t^{\alpha /2} \left(\|u\|_{L^\infty }^\alpha \|\mathcal{L} ^2u\|_{L^2}+\|u\|_{L^\infty }^{\alpha -1}\|\mathcal{L} u\|_{L^4}^2 \right)\notag\\
		&\lesssim &   t^{\alpha /2} \|u\|_{L^\infty } ^\alpha  \|\mathcal{L} ^2u\|_{L^2},\notag 
	\end{eqnarray}
	which together with (\ref{4103})--(\ref{3.4-0})  yields the desired estimate (\ref{asd2}).

	We now resume the proof of Lemma \ref{l4.5}. Using the same method as that used to derive  (\ref{b1}), one obtains
	\begin{equation}
		\left\|G_h^w(1-\Gamma)(|v|^\alpha v)\right\|_{L^\infty }\le Ch^{1/4}\|\widetilde{\mathcal{L}}(|v|^\alpha v)\|_{L^2}\notag
	\end{equation} 
	\begin{equation}
		\left\|G_h^w(1-\Gamma)(|v|^\alpha v)\right\|_{L^2 }\le Ch^{1/2}\|\widetilde{\mathcal{L}}(|v|^\alpha v)\|_{L^2},\notag
	\end{equation}
	which together with (\ref{asd1}) proves the first part of Lemma \ref{l4.5}. 
	
	When  $u_0\in H^{0,2}$, we  write 
	\begin{equation}
		1-\Gamma(x,\xi)=h\Gamma_{-4}(\frac{x+F'(\xi)}{\sqrt h})(\frac{x+F'(\xi)}{ h})^2,\notag
	\end{equation}
	where $\Gamma_{-4}(\xi)=\frac{1-\gamma(\xi)}{\xi^2}$, satisfying $|\partial_\xi^\alpha \Gamma_{-4}(\xi)|\le C_\alpha <\xi>^{-\alpha }$ for any $\alpha \in \mathbb{N}$. Then using Lemma \ref{l0}, and the definition of $\widetilde{\mathcal{L}}$ in (\ref{3.4}), one gets 
	\begin{eqnarray}
		&&G_h^w(1-\Gamma)(|v|^\alpha v)= h G_h^w(\Gamma_{-4}(\frac{x+F'(\xi)}{\sqrt h}))\circ \widetilde{\mathcal{L}}^2(|v|^\alpha v).\notag
	\end{eqnarray}
	An application of  Proposition \ref{P2.6} yields,   for all $t\ge1$,  
	\begin{equation}
		\left\|G_h^w(1-\Gamma)(|v|^\alpha v)\right\|_{L^\infty }\le Ch^{3/4}\|\widetilde{\mathcal{L} }^2(|v|^\alpha v)\|_{L^2},\notag
	\end{equation} 
	\begin{equation}
		\left\|G_h^w(1-\Gamma)(|v|^\alpha v)\right\|_{L^2 }\le C h\|\widetilde{\mathcal{L} }^2(|v|^\alpha v)\|_{L^2}.\notag
	\end{equation}
	This  together with (\ref{asd2}) proves   the second part  of  Lemma \ref{l4.5}. 
\end{proof}

Using  similar argument as  in the proof of Lemma \ref{l4.5}, we  obtain the following lemma  easily and omit the details.
\begin{lem}\label{l5.6}
	If  $u_0\in H^{0,1}$, then  for all $t\ge1$,
	\begin{equation}
		\||v|^\alpha v-|v_\Lambda|^\alpha v_\Lambda\|_{L^\infty}\le C (\left\|v\right\|_{L^\infty }^\alpha+\left\|v_{\Lambda }\right\|_{L^\infty }^\alpha )\|\widetilde{\mathcal{L}}v\|_{L^2}t^{-1/4},\notag
	\end{equation}
	\begin{equation}
		\||v|^\alpha v-|v_\Lambda|^\alpha v_\Lambda\|_{L^2} \le C (\left\|v\right\|_{L^\infty }^\alpha+\left\|v_{\Lambda }\right\|_{L^\infty }^\alpha ) \|\widetilde{\mathcal{L}}v\|_{L^2}t^{-1/2}. \notag
	\end{equation}
	Furthermore, if $u_0\in H^{0,2}$ and $\alpha \ge1$, then for all $t\ge1$, 	\begin{equation}
		\||v|^\alpha v-|v_\Lambda|^\alpha v_\Lambda\|_{L^\infty}\le C (\left\|v\right\|_{L^\infty }^\alpha+\left\|v_{\Lambda }\right\|_{L^\infty }^\alpha )\|\widetilde{\mathcal{L} }^2v\|_{L^2}t^{-3/4},\notag
	\end{equation}
	\begin{equation}
		\||v|^\alpha v-|v_\Lambda|^\alpha v_\Lambda\|_{L^2} \le C (\left\|v\right\|_{L^\infty }^\alpha+\left\|v_{\Lambda }\right\|_{L^\infty }^\alpha )  \|\widetilde{\mathcal{L} }^2v\|_{L^2}t^{-1}. \notag
	\end{equation}
\end{lem}
\begin{proof}
	[\textbf{Proof of Proposition \ref{lp} }]
	It follows from  (\ref{3234}), (\ref{0451}) and (\ref{4104})  that for all $t\ge1$, 
	\begin{equation}
		\label{521}
		\|\widetilde{\mathcal{L} }v(t,\cdot)\|_{L^2}\le \|u_0\|_{H^{0,1}},
	\end{equation}
	\begin{equation}
		\label{4105}
		\|\widetilde{\mathcal{L} }^2v(t,\cdot)\|_{L^2}\le C(\left\|u_0\right\|_{H^{0,2}}+\|u_0\|_{H^{0,2}}^{2\alpha +1})t^{2-\alpha }.
	\end{equation}
	Proposition \ref{lp} then follows from  (\ref{521})-(\ref{4105}) and Lemmas \ref{l4.3}--\ref{l5.6}. 
\end{proof}

\subsection{The rough $L^\infty $ estimate for $v_{\Lambda }$}\label{sub2}
The goal of this  subsection is to derive the $L^\infty $ estimate for $v_{\Lambda }$ from the ODE  (\ref{3.4-000}) and therefore completing  the proof of Theorem \ref{T1.2}. 
Notice that  the estimate of  $R_2(v)$ depends on  $\|v_{\Lambda }\|_{L^\infty }$ and the initial value is large,  we apply  the  bootstrap argument to derive the decay estimate (\ref{d1}) with  $K$ large. 

Let $K$ be  a sufficiently large  constant such that 
\begin{eqnarray}
	\label{4121} 
	K>\max \left\{2^{1/2+1/\alpha },\ 1\right\} 
\end{eqnarray}
\begin{eqnarray}\label{3201}
	\frac{2-\alpha }{2}K^{-\alpha }+C_2\alpha 2^{\alpha +2}K^{-1} < \alpha \lambda _2,
\end{eqnarray}
where $C_1,\ C_2$ are  constants  depending on  $u_0$ that appear in (\ref{e1}), (\ref{e5}) respectively.

We assume that $v$ satisfies a bootstrap hypotheses on $t\in[1,T_1]$:
\begin{eqnarray}\label{bass}
	\left\|v_{\Lambda}(t,x)\right\|_{L_x^\infty }\le 	2	K t^{1/2-1/\alpha }.
\end{eqnarray}
From  Lemma \ref{l4.3-00}, (\ref{3.4-0})  and the local decay estimate  (\ref{ldecay}), we see that, for $t\in[1,2]$, 
\begin{eqnarray}\label{e1}
	\left\|v_{\Lambda }(t,x)\right\|_{L_x^\infty }&\le&	\left\|v(t,x)\right\|_{L_x^\infty }+	\left\|v_{\Lambda^c }(t,x)\right\|_{L_x^\infty }\notag\\
	&\le&C_1(1+t^{-1/4})	<	K t^{1/2-1/\alpha }.
\end{eqnarray}
This implies  that $T_1>2$. Moreover, it follows from  (\ref{bass}) and Proposition \ref{lp} that $\lambda t^{-\alpha /2}  |v_\Lambda|^\alpha v_\Lambda+ t^{-\alpha /2} \left(R_1(v)+R_2(v)\right)$ is integrable on $(1,T_1 )$; so that $v_{\Lambda}(t,x)\in $ $C((1,T_1),$ $L^\infty )$ by the equation  (\ref{4.34}). 

The following lemma is crucial to close the bootstrap hypotheses (\ref{bass}). 
\begin{lem}\label{l4.8}
	Under the assumptions (\ref{bass}) and  $4/3<\alpha <2$, we have that, for all $t\in (1,T_1)$, 
\begin{equation}
		\left\|v_{\Lambda }(t,x)\right\|_{L_x^\infty }\le		K t^{1/2-1/\alpha }.\notag
\end{equation}
\end{lem}
\begin{proof}The proof is inspired by Lemma 2.3 of \cite{Kita}.	We prove it by contradiction argument.  Assume there exists some $(t_0,\xi_0)\in(1,T_1)\times \mathbb{R}$  such that $\left|v_{\Lambda }(t_0,\xi_0)\right|>Kt_0^{1/2-1/\alpha }$.
	From (\ref{e1})  and the continuity of $v_{\Lambda }(t,\xi_0)$, we can find  $t_*\in(1,t_0)$ such that 
	\begin{equation}
		\label{492}
		\left|v_{\Lambda }(t,\xi_0)\right|>Kt^{1/2-1/\alpha },\qquad \text{holds for all } t_*<t\le t_0
	\end{equation}
	and  furthermore $\left|v_{\Lambda }(t_*,\xi_0)\right|=Kt_*^{1/2-1/\alpha }$. Multiplying $\left|v_{\Lambda }(t,\xi_0)\right|^{-(\alpha +2)}\overline{v_{\Lambda }(t,\xi_0)}$ on both hand sides of (\ref{4.34}) and taking the imaginary part, we obtain 
	\begin{equation}\label{e2}
		-\frac{1}{\alpha }\frac{d}{dt}\left|v_{\Lambda }\right|^{-\alpha }
		=- \lambda_2 t^{-\alpha /2}-\text{Im} \left(t^{-\alpha /2}\left(R_1(v)+R_2(v)\right)\overline{v_{\Lambda }}\right)\left|v_{\Lambda }\right|^{-(\alpha +2)}.
	\end{equation}
	On the other hand, from Proposition \ref{lp},  we have 
	\begin{eqnarray}\label{e4}
		\left\|R_1(v)+R_2(v)\right\|_{L^\infty }&\le& C\left(t^{-5/4+\alpha /2}+(\left\|v_{\Lambda}\right\|_{L^\infty }^\alpha +\left\|v_{\Lambda^c}\right\|_{L^\infty }^\alpha)t^{-1/4}\right)\notag\\
		&\le& C2^{\alpha +2}K^\alpha t^{-5/4+\alpha /2},
	\end{eqnarray}
	for all $t\in (t_*,t_0)$, where we used (\ref{4121}),  Lemma \ref{l4.3-00} and the bootstrap hypotheses (\ref{bass})  to bound  $\|v_{\Lambda }\|_{L^\infty }^\alpha +\|v_{\Lambda ^c}\|_{L^\infty }^\alpha $.  It then  follows from  (\ref{492})--(\ref{e4})  that there exists $C_2>0$ such that for  $t_*<t<t_0$
	\begin{eqnarray} \label{e5}
		-\frac{1}{\alpha }\frac{d}{dt}\left|v_{\Lambda }(t,\xi_0)\right|^{-\alpha }\le -\lambda _2t^{-\alpha /2}+C_22^{\alpha +2}K^{-1} t^{-3/4-\alpha /2+1/\alpha }.
	\end{eqnarray}
	Integrating (\ref{e5}) from $t_*$ to $t$, we get
	\begin{eqnarray}
		&&\left|v_{\Lambda }(t,\xi_0)\right|^{-\alpha }-\left|v_{\Lambda }(t_*,\xi_0)\right|^{-\alpha }\notag\\
		&\ge&\frac{2\alpha \lambda _2}{2-\alpha }\left(t^{1-\alpha /2}-t_*^{1-\alpha /2}\right)-C_2\alpha 2^{\alpha +2}K^{-1} \int_{t^*}^{t}s^{-3/4-\alpha /2+1/\alpha } \mathrm{d}s .\notag
	\end{eqnarray}
	This inequality  together with $\left|v_{\Lambda }(t_*,\xi_0)\right|=Kt_*^{1/2-1/\alpha }$ implies that 
	\begin{eqnarray}
			\left(t^{1/\alpha -1/2}\left|v_{\Lambda }(t,\xi_0)\right|\right)^{-\alpha }
		&\ge& \left(\frac{t_*}{t}\right)^{1-\alpha /2}K^{-\alpha} +\frac{2\alpha \lambda _2}{2-\alpha }\left(1-\left(\frac{t_*}{t}\right)^{1-\alpha /2}\right)\notag\\
		&&-C_2\alpha 2^{\alpha +2}K^{-1} t^{-1+\alpha /2}\int_{t^*}^{t}s^{-3/4-\alpha /2+1/\alpha } \mathrm{d}s\notag\\
		&=:&f(t).\notag
	\end{eqnarray}
	Note that $f(t_*)=K^{-\alpha }$ and $f(t)$ is monotone increasing around $t=t_*$. Indeed, by (\ref{3201})
	\begin{equation}
			f'(t_*)=\left(\frac{\alpha -2}{2}K^{-\alpha  }+\alpha \lambda _2-C_2\alpha 2^{\alpha +2}K^{-1} t_*^{-3/4+1/\alpha  }\right)t_*^{-1}>0. \notag
	\end{equation}
	Thus, if $t$ is slightly larger than $t_*$, then $\left(t^{1/\alpha -1/2}\left|v_{\Lambda }(t,\xi_0)\right|\right)^{-\alpha }>K^{-\alpha }$, which contradicts (\ref{492}), from which Lemma \ref{l4.8} follows. 
\end{proof}
\begin{proof}
	[\textbf{Proof of Theorem \ref{T1.2}}]
	 Lemma \ref{l4.8} and a standard continuation argument  imply $T_1=\infty $ and   that for all $t\ge1$, 
\begin{equation}
	\label{d1}
	\left\|v_{\Lambda}(t,x)\right\|_{L_x^\infty }\le 	2	K t^{1/2-1/\alpha }.
\end{equation}
This together with  (\ref{3.4-0}) and Lemma \ref{l4.3-00} yields  $\|u(t,x)\|_{L^\infty _x}\lesssim  t^{-1/\alpha }$, from which Theorem \ref{T1.2} follows. 
\end{proof}	 
\section{The proof of Theorem \ref{T1.3}} \label{S6}
In this section, we prove Theorem \ref{T1.3}. We only consider  the nontrivial case $u_0\neq0$. To establish the asymptotic formula for the solutions for a wider range of $\alpha $, we first need to refine the $L^\infty $ estimate of $v_{\Lambda}$. 
\subsection{The refined $L^\infty $ estimate for $v_{\Lambda }$}
In this subsection, we show how to obtain the refined  $L^\infty $ estimate for  $v_{\Lambda }$ (Lemma \ref{l5.1}). This will be essential in proving the large time asymptotics of the solutions in Subsection \ref{sub3}.

Substituting the rough  $L^\infty $ bound (\ref{d1}) of  $v_{\Lambda }$ into Proposition \ref{lp}, we can rewrite the equation (\ref{4.34}) as 
\begin{equation}
	D_tv_\Lambda=w(x)v_{\Lambda }+\lambda t^{-\alpha /2}  |v_\Lambda|^\alpha v_\Lambda+ t^{-\alpha /2} R(v), \label{4.340}
\end{equation}
where $R(v)$ satisfies, for all $t\ge1$, \\
(a) $u_0\in H^{0,1}$: 
\begin{equation}
	\|R(v)\|_{L^\infty}\leq Ct^{-\frac{5}{4}+\alpha /2},\label{r2}
\end{equation}
\begin{equation}
	\|R(v)\|_{L^2}\leq Ct^{-3/2+\alpha /2}.\label{r1}
\end{equation}
(b) $u_0\in H^{0,2}$:
\begin{equation}
	\|R(v)\|_{L^\infty}\leq Ct^{1/4-\alpha/2 },\label{r3}
\end{equation}
\begin{equation}
	\|R(v)\|_{L^2}\leq Ct^{-\alpha/2 }.\label{r4}
\end{equation}
For $0<\varepsilon<2\alpha \lambda _2$, we define
\begin{eqnarray}
	\begin{cases}
		h_1(\varepsilon ):=\frac{3}{4}-\frac{\alpha}{2} +\lambda _2\frac{2-\alpha }{2\alpha\lambda _2-\varepsilon  },\ \text{when }u_0\in H^{0,1},  \frac{1+\sqrt{33}}{4}<\alpha <2,\\
		h_2(\varepsilon ):=\frac{9}{4}-\frac{3\alpha}{2} +\lambda _2\frac{2-\alpha }{2\alpha\lambda _2-\varepsilon  },\ \text{when }u_0\in H^{0,2}, \frac{7+\sqrt{145}}{12}<\alpha <2.
	\end{cases}\notag
\end{eqnarray}
Then by direct computation,  we have 
\begin{eqnarray}
	\begin{cases}
		h_1(0)=\frac{3}{4}-\frac{\alpha}{2} +\frac{2-\alpha }{2\alpha }<0,\ \text{when }u_0\in H^{0,1},  \frac{1+\sqrt{33}}{4}<\alpha <2,\\
		h_2(0)=\frac{9}{4}-\frac{3\alpha}{2} +\frac{2-\alpha }{2\alpha }<0, \ \text{when }u_0\in H^{0,2}, \frac{7+\sqrt{145}}{12}<\alpha <2.
	\end{cases}
\end{eqnarray}
By the continuity of $h_1,\ h_2$, we can find $0<\varepsilon _1<2\alpha \lambda _2$ such that for any $0<\varepsilon _0<\varepsilon _1$
\begin{equation}
	\frac{3}{4}-\frac{\alpha}{2} +\lambda _2\frac{2-\alpha }{2\alpha \lambda _2-\varepsilon_0 }<0,\ \text{when }u_0\in H^{0,1},  \frac{1+\sqrt{33}}{4}<\alpha <2,\label{3221}
\end{equation}
\begin{equation}
	\frac{9}{4}-\frac{3\alpha}{2} +\lambda _2\frac{2-\alpha }{2\alpha \lambda _2-\varepsilon_0 }<0,\ \text{when }u_0\in H^{0,2}, \frac{7+\sqrt{145}}{12}<\alpha <2.\label{4241}
\end{equation}
For  $0<\varepsilon _0<\varepsilon _1$, we define 
\begin{eqnarray}\label{3211}
	K_0=\left(\frac{2-\alpha }{2\alpha \lambda _2-\varepsilon_0 }\right)^{1/\alpha },\qquad T_0=\max \left\{\left(\frac{2C_3\alpha  }{\varepsilon _0K_0^{\alpha +1}  }\right)^{\frac{4\alpha }{3\alpha -4}},\ e\right\} ,
\end{eqnarray}
where $ C_3>0$ is the constant in  (\ref{e3}) that depends on  $u_0$. 

The following  refined $L^\infty $ estimate for $v_{\Lambda }$ is true:
\begin{lem}\label{l5.1}
	Assume $u_0\in H^{0,1},\  \frac{1+\sqrt{33}}{4}<\alpha <2$ or $u_0\in H^{0,2},\ \frac{7+\sqrt{145}}{12}<\alpha <2$.  There exists $T^*(\varepsilon _0)>T_0$, such that for all $t>T^*(\varepsilon _0)$
	\begin{eqnarray}
		\left\|v_{\Lambda }(t,x)\right\|_{L_x^\infty }\le 	K_0t^{1/2-1/\alpha }.\notag
	\end{eqnarray}
\end{lem}
\begin{proof}	
	In what follows, we   prove Lemma \ref{l5.1} in the case $u_0\in H^{0,1},\  \frac{1+\sqrt{33}}{4}<\alpha <2$ by the contradiction argument. The case $u_0\in H^{0,2}$, $\frac{7+\sqrt{145}}{12}<\alpha <2$ follows from a similar argument, but using (\ref{r3})--(\ref{r4}), (\ref{4241}) instead of (\ref{r2})--(\ref{r1}), (\ref{3221}).  
	
	Assume by contradiction that  there exist  $\{(t_n,\xi_n)\}_{n=1}^\infty \subset (T_0,\infty )\times \mathbb{R}$ with $t_n$ monotone increases to $\infty $ such that 
	\begin{equation}\label{4111}
		\left|v_{\Lambda }(t_n,\xi_n)\right|>K_0t_n^{1/2-1/\alpha },\qquad \text{for all }n\ge1.
	\end{equation}
	We first claim that, for every fixed $n\ge1$, we have  
	\begin{eqnarray}\label{3212}
		|v_{\Lambda }(t,\xi_n)|>
		K_0t^{1/2-1/\alpha }, \qquad \text{for all } t\in (T_0,t_n).
	\end{eqnarray}
	In fact, if this claim is not true,  there would exist some $t^*_n\in(T_0,t_n)$ such that 
	\begin{equation}\label{4112}
		\left|v_{\Lambda }(t,\xi_n)\right|>K_0t^{1/2-1/\alpha }, \qquad \text{holds for all } t\in (t^*_n,t_n),
	\end{equation} 
	and furthermore $\left|v_{\Lambda }(t^*_n,\xi_n)\right|=K_0(t^*_n)^{1/2-1/\alpha }$. Multiplying $\left|v_{\Lambda }(t,\xi_n)\right|^{-(\alpha +2)}\overline{v_{\Lambda }(t,\xi_n)}$ on both hand sides of (\ref{4.340}) and taking the imaginary part, we obtain, for all $t\in (t_n^*,t_n)$,  
	\begin{eqnarray}
		-\frac{1}{\alpha }\frac{d}{dt}\left|v_{\Lambda }(t,\xi_n)\right|^{-\alpha }&=&- \lambda_2 t^{-\alpha /2}-\text{Im} \left(t^{-\alpha /2}R(v)\overline{v_{\Lambda }}\right)\left|v_{\Lambda }\right|^{-(\alpha +2)}.\label{13}
	\end{eqnarray}
	Using (\ref{r2}) and (\ref{4112}) to bound the third term in (\ref{13}),  we deduce that there exists $C_3>0$ such that for all $t\in (t_n^*,t_n)$,
	\begin{equation}
		-\frac{1}{\alpha }\frac{d}{dt}\left|v_{\Lambda }(t,\xi_n)\right|^{-\alpha }\le -\lambda _2  t^{-\alpha /2}+C_3 K_0^{-(\alpha +1)}t^{-3/4-\alpha /2+1/\alpha  }.\label{e3}
	\end{equation}
	Integrating the above inequality from $t_n^*$ to $t$ and using $\left|v_{\Lambda }(t_n^*,\xi_0)\right|=K_0(t_n^*)^{1/2-1/\alpha }$, we obtain 
	\begin{eqnarray} 
			\left(t^{1/\alpha -1/2}\left|v_{\Lambda }(t,\xi_n)\right|\right)^{-\alpha }&\ge& \left(\frac{t_n^*}{t}\right)^{1-\alpha /2}K_0^{-\alpha} +\frac{2\alpha \lambda _2}{2-\alpha }\left(1-\left(\frac{t_n^*}{t}\right)^{1-\alpha /2}\right)\notag\\
		&&-C_3 \alpha K_0^{-(\alpha +1)}t^{-1 +\alpha /2 }\int_{t^*}^{t}s^{-3/4-\alpha /2+1/\alpha } \mathrm{d}s\notag\\
		&=:&g(t).\notag
	\end{eqnarray}
	We see that $g(t_n^*)=K_0^{-\alpha }$ and $g(t)$ is monotone increasing around $t=t_n^*$. Indeed, 
	\begin{eqnarray}
		g'(t_n^*)=\left(\frac{\alpha -2}{2}K_0^{-\alpha  }+\alpha \lambda _2-C_3 \alpha K_0^{-(\alpha +1)}(t_n^*)^{-3/4+1/\alpha  }\right)(t_n^*)^{-1}>0,\notag
	\end{eqnarray}
	where we used  (\ref{3211}). 
	Thus, if $t$ is slightly lager than $t_n^*$, then $\left|t^{1/\alpha -1/2}v_{\Lambda}(t,\xi_n)\right|^{-\alpha }>K_0^{-\alpha }$, which contradicts (\ref{4112}).   Thus we finish the proof of Claim (\ref{3212}). 
	
	In what follows, we use Claim (\ref{3212}) to derive a contradiction to (\ref{4111}). Notice that (\ref{e3}) holds for all $t\in (T_0,t_n)$ by a similar argument used before, but using Claim (\ref{3212}) instead of (\ref{4112}).  	Integrating (\ref{e3}) from $T_0$ to $t_n$,  using (\ref{r2}) and Claim (\ref{3212}),   we have, for every $n\ge1$,
	\begin{eqnarray}\label{3261}
			\left(t_n^{1/\alpha -1/2}\left|v_{\Lambda }(t_n,\xi_n)\right|\right)^{-\alpha }
		&\ge& \frac{|v_\Lambda (T_0,\xi_n)|^{-\alpha }}{t_n^{1-\alpha /2}} +\frac{2\alpha \lambda _2}{2-\alpha }\left(1-\left(\frac{T_0}{t_n}\right)^{1-\alpha /2}\right)\notag\\
		&&-C_3\alpha K_0^{-(\alpha +1)}t_n^{-1 +\alpha /2 }\int_{T_0}^{t_n}s^{-3/4-\alpha /2+1/\alpha } \mathrm{d}s.
	\end{eqnarray}
	Since $K_0^{-\alpha }\ge 	(t_n^{1/\alpha -1/2}\left|v_{\Lambda }(t_n,\xi_n)\right|)^{-\alpha }$ by (\ref{4111}), and $\frac{|v_{\Lambda}(T_0,\xi_n)|^{-\alpha }}{t_n^{1-\alpha /2}}\rightarrow 0$ as $n\rightarrow \infty $ by Claim (\ref{3212}), we can let $n\rightarrow \infty $ in (\ref{3261}) and  obtain
	\begin{eqnarray}
		K_0^{-\alpha }\ge \frac{2\alpha \lambda _2}{2-\alpha }.\notag
	\end{eqnarray}
	This contradicts the definition of $K_0$ in (\ref{3211}),  and thus completing the proof of Lemma \ref{l5.1}.
\end{proof}
\subsection{The proof of Theorem \ref{T1.3}}\label{sub3}
We are now in a position to prove Theorem \ref{T1.3}.   We give only the proof  for the case $u_0\in H^{0,1}$, $ \frac{1+\sqrt{33}}{4}<\alpha <2$.  The case $u_0\in H^{0,2}$, $\frac{7+\sqrt{145}}{12}<\alpha <2$ follows from a similar argument, but using (\ref{r3})--(\ref{r4}), (\ref{4241})  instead of (\ref{r2})--(\ref{r1}), (\ref{3221}). The proof is inspired by Section 5 of \cite{CPDE}. 

\begin{proof}
	 [\textbf{Proof of part (a)}]
	Assume $T^*(\varepsilon _0)$ is the constant in Lemma \ref{l5.1}, $\Phi(t,x)$, $ K(t,x)$, $\psi_+(x)$, $S(t,x)$ are the functions defined in (\ref{eq1})--(\ref{eq2}), respectively. 
	From the definition of  $\Phi(t,x)$ in (\ref{eq1}) and Lemma \ref{l5.1}, we have 
	\begin{eqnarray}\label{3224}
		\left\|\Phi(t,x)\right\|_{L_x^\infty }&\le& \int_{1}^{T^*(\varepsilon _0)}s^{-\alpha /2}\|v_{\Lambda }(s,x)\|_{L^\infty _x}^\alpha  \mathrm{d}s+ K_0^\alpha \int_{T^*(\varepsilon _0)}^{t}s^{-1} \mathrm{d}s\notag\\
		&\le& CT^*(\varepsilon _0)+ K_0^\alpha \log t,\label{3195}
	\end{eqnarray}
	for all $t\ge1$, where we used (\ref{d1}) to estimate the first integral. 
	
	Set 
	\begin{equation}\label{4107}
		z(t,x)=v_\Lambda(t,x) e^{-i(w(x)t+\lambda \Phi(t,x))}, \qquad  t>1.
	\end{equation}
	From the equation (\ref{4.340}) and (\ref{4107}), we have that for $t>1$,
	\begin{eqnarray}
		\partial_t  z(t,x)=\frac{iR(v)}{t^{\alpha /2}}e^{-i(w(x)t+\lambda\Phi(t,x))};\notag
	\end{eqnarray}
	so that for all $t_2>t_1>1$,
	\begin{eqnarray}
		z(t_2,x)-z(t_1,x)=i\int_{t_1}^{t_2}s^{-\alpha /2}R(v)e^{-i(w(x)s+\lambda\Phi(s,x))} \mathrm{d}s. \label{7}
	\end{eqnarray}
	Since $-1/4+\lambda _2K_0^\alpha <0$ by (\ref{3221}) and (\ref{3211}), it follows  from (\ref{r2})--(\ref{r1}), (\ref{3195}) and (\ref{7}) that
	\begin{eqnarray}\label{3303}
		\left\|z(t_2,x)-z(t_1,x)\right\|_{L_x^\infty \cap L_x^2}&\lesssim& \int_{t_1}^{t_2}s^{-\alpha /2}\left\|R(v)\right\| _{ L_x^\infty \cap L_x^2}e^{\lambda _2\left\|\Phi(s,x)\right\|_{L_x^\infty }}\mathrm{d}s \notag\\
		&\lesssim_{\varepsilon _0} &  \int_{t_1}^{t_2}s^{-5/4+\lambda _2K_0^\alpha } \mathrm{d}s\notag\\
		&\lesssim_{\varepsilon _0} &  t_1^{-1/4+\lambda _2K_0^\alpha },
	\end{eqnarray}
	holds for any $t_2>t_1>1$. 
	Thus  there exists $z_+(x)\in L_x^2\cap L_x^\infty $ such that 
	\begin{eqnarray}\label{3196}
		\left\|z(t,x)-z_+(x)\right\|_{L_x^2\cap L_x^\infty }\lesssim_{\varepsilon _0} t^{-1/4+\lambda _2K_0^\alpha }.\label{1.12}
	\end{eqnarray}
	This finishes the proof of part (a). 
\end{proof}
\begin{proof}
 [\textbf{Proof of part (b)}] The first step is to  derive the asymptotic formula  (\ref{3222}) for $\Phi(t,x)$.  Note that 
 \begin{eqnarray}
 	\partial_t  \Phi(t,x)=t^{-\alpha /2}\left|v_\Lambda(t,x) \right|^\alpha =t^{-\alpha /2}\left|z(t,x)\right|^\alpha e^{-\alpha\lambda _2 \Phi(t,x)}\notag
 \end{eqnarray}
 for all  $t>1$ by the definition of  $z(t,x)$ in (\ref{4107}); so that
 \begin{eqnarray}
 	\partial_t  e^{\alpha \lambda _2\Phi(t,x)}=\alpha \lambda _2t^{-\alpha /2}\left|z(t,x)\right|^{\alpha }.\notag
 \end{eqnarray}
 Integrating the above equation from $1$ to $t$, we get
 \begin{eqnarray}\label{4}
 	e^{\alpha \lambda _2\Phi(t,x)}=1+\alpha \lambda _2\int_{1}^{t}s^{-\alpha /2}\left|z(s,x)\right| ^{\alpha }\mathrm{d}s.
 \end{eqnarray}
 From (\ref{3301}), (\ref{a2}) and (\ref{4}), we have  
 \begin{eqnarray}\label{3197}
 	e^{\alpha \lambda _2\Phi(t,x)}-K(t,x)-\psi_+(x)=- \alpha \lambda_2 \int_{t}^{\infty }s^{-\alpha /2}\left(\left|z(s,x)\right|^\alpha -\left|z_+(x)\right|^\alpha \right) \mathrm{d}s.
 \end{eqnarray}
 Since $\left|\left|u\right|^\alpha -\left|v\right|^\alpha \right|\lesssim (\left|u\right|^{\alpha -1}+\left|v\right|^{\alpha -1})\left|u-v\right|$, and $z(s,x),\ z_+(x)\in L_x^\infty $, we deduce from (\ref{3196}) and (\ref{3197}) that, for all $t\ge 1$,  
 \begin{eqnarray}\label{3222}
 	\left\|e^{\alpha \lambda _2\Phi(t,x)}-K(t,x)-\psi_+(x)\right\|_{L_x^\infty \cap L_x^2}\lesssim_{\varepsilon _0} \int_{t}^{\infty }s^{-1/4-\alpha/2 +\lambda _2K_0^\alpha } \mathrm{d}s\lesssim t^{-\beta},
 \end{eqnarray}
 where $\beta=-(3/4-\alpha/2 +\lambda _2K_0^\alpha )>0$ by (\ref{3221}) and (\ref{3211}).
 
 We now prove the asymptotic formula  (\ref{3222}).  Since  $|e^{iw(x)t+i\lambda \Phi(t,x)}|=e^{-\lambda _2\Phi(t,x)}\le1$,  we have 
 \begin{eqnarray}\label{3304}
 	&&\|e^{i(w(x)t+\lambda S(t,x))}z_+(x)-v_{\Lambda }(t,x)\|_{L_x^\infty\cap L_x^2 }\notag\\
 	&\le& \|e^{iw(x)t}(e^{i\lambda S(t,x)}-e^{i\lambda \Phi(t,x)})z_+\|_{L_x^\infty\cap L_x^2 }+\|e^{iw(x)t}e^{i\lambda \Phi(t,x)}z_+-v_{\Lambda }(t,x)\|_{L_x^\infty\cap L_x^2 }\notag\\
 	&\le&\|e^{i\lambda _1S(t,x)}(e^{-\lambda _2S(t,x)}-e^{-\lambda _2\Phi(t,x)})z_+(x)\|_{L_x^\infty\cap L_x^2 }\notag\\
 	&&+\|(e^{i\lambda _1S(t,x)}-e^{i\lambda _1\Phi(t,x)})e^{-\lambda _2\Phi(t,x)}z_+(x)\|_{L_x^\infty \cap L_x^2}\notag\\
 	&&\qquad +\|z_+(x)-z(t,x)\|_{L_x^\infty\cap L_x^2 }.
 \end{eqnarray}
 Note that  $K(t,x)+\psi_+(x)\ge1/2$ for $t$ sufficiently large by (\ref{3222}); so that 
 \begin{eqnarray}\label{3225}
 	&&\|e^{i\lambda _1S(t,x)}(e^{-\lambda _2S(t,x)}-e^{-\lambda _2\Phi(t,x)})z_+(x)\|_{L_x^\infty\cap L_x^2 }\notag\\
 	&=& \|(K(t,x)+\psi_+(x))^{-1/\alpha }-e^{-\lambda _2\Phi(t,x)}\|_{L_x^\infty\cap L_x^2}\notag\\
 	&\lesssim &\|e^{\lambda _2\Phi(t,x)}-(K(t,x)+\psi_+(x))^{1/\alpha }\|  _{L_x^\infty \cap L_x^2}\notag\\
 	&\lesssim &  \|e^{\alpha \lambda _2\Phi(t,x)}-(K(t,x)+\psi_+(x))\|^{1/\alpha }  _{L_x^\infty\cap L_x^2 }\lesssim_{\varepsilon _0} t^{-\beta /\alpha }.
 \end{eqnarray}
 Similarly,  we have   
 \begin{eqnarray}\label{3305}
 	&&\|(e^{i\lambda _1S(t,x)}-e^{i\lambda _1\Phi(t,x)})e^{-\lambda _2\Phi(t,x)}z_+(x)\|_{L_x^\infty \cap L_x^2}\notag\\
 	&\lesssim &  \|S(t,x)-\Phi(t,x)\|_{L_x^\infty \cap L_x^2}\notag\\
 	&\lesssim &  \|e^{\lambda _2S(t,x)}-e^{\lambda _2\Phi(t,x)}\|_{L_x^\infty \cap L_x^2}\notag\\
 	&= & \|e^{\lambda _2\Phi(t,x)}-(K(t,x)+\psi_+(x))^{1/\alpha }\|  _{L_x^\infty \cap L_x^2}
 	\lesssim_{\varepsilon _0}   t^{-\beta /\alpha }.  
 \end{eqnarray}
 Substituting  (\ref{3303}), (\ref{3225})--(\ref{3305}) into (\ref{3304}),  we get 
 \begin{eqnarray}\label{3306}
 	\|e^{i(w(x)t+\lambda S(t,x))}z_+(x)-v_{\Lambda }(t,x)\|_{L_x^\infty\cap L_x^2 }\lesssim_{\varepsilon _0} t^{-\gamma},
 \end{eqnarray}
 for $t$ sufficiently large, where $0<\gamma:=\min \{1/4-\lambda _2K_0^\alpha, \ \beta/\alpha \}<1/4$.  It then  follows from Lemma \ref{l4.3-00}, (\ref{4.2}) and (\ref{3306}) that the asymptotic formula (\ref{3302}) holds. 
\end{proof}

\begin{proof}
	[\textbf{Proof of part (c)}] From the  asymptotic formula (\ref{3302}), we have 
	\begin{eqnarray}
		&&e^{-iF(D)t}e^{-i\lambda S(t,\frac{x}{t})}u(t,x)\notag\\
		&=&e^{-iF(D)t}\frac{1}{\sqrt t}z_+(\frac{x}{t})e^{iw(\frac{x}{t})t}+O_{L^2}(t^{-\gamma})\notag\\
		&=& \frac{1}{2\pi} \iint e^{i(x-y)\xi}e^{-iF(\xi)t}\frac{1}{\sqrt t}z_+(\frac{y}{t})e^{iw(\frac{y}{t})t}dyd\xi+O_{L^2}(t^{-\gamma})\notag\\
		&=& \frac{\sqrt t}{2\pi} \iint e^{ix\xi-it(y\xi+F(\xi))}z_+(y)e^{itw(y)}dyd\xi+O_{L^2}(t^{-\gamma})\notag\\
		&=& \frac{\sqrt t}{2\pi }\iint e^{ix\xi-itc_2(\xi+\frac{y+c_1}{2c_2})^2}z_+(y)dyd\xi+O_{L^2}(t^{-\gamma})\label{4252}
	\end{eqnarray}
	as $t\rightarrow \infty $, where we use $y\xi+F(\xi)=w(y)+c_2(\xi+\frac{y+c_1}{2c_2})^2$  in the last step. 
	Moreover, making a change of variables and then using the dominated convergence theorem, we obtain 
	\begin{eqnarray}
		&&\lim_{t\rightarrow \infty }\frac{\sqrt t}{2\pi }\iint e^{ix\xi-itc_2(\xi+\frac{y+c_1}{2c_2})^2}z_+(y)dyd\xi\notag\\
		&=& \lim_{t\rightarrow \infty } \frac{1}{2\pi} \iint e^{ix(\frac{\zeta}{\sqrt t}-\frac{y+c_1}{2c_2})}e^{-ic_2\zeta^2}z_+(y)dyd\zeta \notag\\
		&=&  \frac{1}{2\pi} \iint e^{-ix\frac{y+c_1}{2c_2}}e^{-ic_2\zeta^2}z_+(y)dyd\zeta \notag\\
		&=&\frac{1}{\sqrt {4\pi c_2}}e^{-i\frac{\pi}{4}}e^{-i\frac{c_1x}{2c_2}} (\mathcal{F} z_+)(\frac{x}{2c_2})\qquad \text{ in } L^2.\label{4253}
	\end{eqnarray}
	The modified scattering formula (\ref{123}) is now an immediate consequence of (\ref{4252}) and (\ref{4253}). 
\end{proof}
\begin{proof}
	[\textbf{Proof of part (d)}]  
	From (\ref{4.2}), Lemma \ref{l4.3-00} and  Lemma \ref{l5.1}, we have 
	\begin{eqnarray}
		t^{1/\alpha }\|u(t,x)\|_{L^\infty _x}&\le & t^{1/\alpha -1/2} \left(\|v_{\Lambda }\|_{L^\infty }+\|v_{\Lambda ^c}\|_{L^\infty }\right)\notag\\
		&\le& \left(\frac{2-\alpha }{2\alpha \lambda _2-\varepsilon _0}\right)^{1/\alpha }+Ct^{1/\alpha -3/4},\ \text{for all } t>T^*(\varepsilon _0).\notag
	\end{eqnarray}
	Since $\varepsilon _0$ can be chosen to be  arbitrarily small, we obtain 
	\begin{eqnarray}
		\limsup_{t\rightarrow \infty }t^{1/\alpha }\|u(t,x)\|_{L^\infty _x}\le\left(\frac{2-\alpha }{2\alpha \lambda _2}\right)^{1/\alpha }.\notag
	\end{eqnarray}
	Therefore, the proof of part (d) reduces to show  that 
	\begin{eqnarray}
		\liminf _{t\rightarrow \infty }t^{1/\alpha }\|u(t,x)\|_{L^\infty _x}\ge \left(\frac{2-\alpha }{2\alpha \lambda _2}\right)^{1/\alpha }. \label{z651}
	\end{eqnarray}
	Assume for a while that we have proved 
	\begin{claim}\label{c123}
		If the limit function $z_+(x)$ in (\ref{3196}) satisfies $z_+(x)=0$ for a.e. $x\in \mathbb{R}$, then we must have $u_0=0$. 	\end{claim}
	\noindent Then since  $u_0\neq 0$, there exists $x_0\in \mathbb{R}$ such that $z_+(x_0)\neq0$. So  by (\ref{3306}) and (\ref{652}), we have 
	\begin{eqnarray}
		t^{1/\alpha-1/2 }\|v_{\Lambda }(t,x)\|_{L^\infty _x}&\ge& \frac{t^{1/\alpha-1/2 }|z_+(x_0)|}{(1+\frac{2\alpha \lambda _2}{2-\alpha }|z_+(x_0)|^\alpha (t^{(2-\alpha )/2}-1)+\psi_+(x_0))^{1/\alpha }}\notag\\
		&&+O_{\varepsilon_0 }(t^{1/\alpha -1/2-\gamma(\varepsilon _0)}),\label{121}
	\end{eqnarray}
	where  $\gamma(\varepsilon _0)=\min \left\{1/4-\lambda _2 K_0^\alpha ,\ \beta/\alpha \right\} $ with  $\beta=-(3/4-\alpha /2+\lambda _2 K_0^\alpha )$ (see (\ref{3222}) and (\ref{3306})). Since 
	$\alpha >\frac{5+\sqrt {89}}{8}$ and  $\lambda _2 K_0^\alpha =\frac{(2-\alpha )\lambda _2}{2\alpha \lambda _2-\varepsilon _0}$, we have  by direct calculation
	\begin{equation}
		\lim _{\varepsilon _0 \rightarrow0}\left(\frac{1}{\alpha }-\frac{1}{2}-\gamma(\varepsilon _0)\right)=\frac{1}{\alpha }-\frac{1}{2}-\min \left\{\frac{1}{4}-\frac{2-\alpha }{2\alpha },\frac{2\alpha ^2-\alpha -4}{4\alpha ^2}\right\} <0.\notag
	\end{equation}
	Therefore, taking  $\varepsilon _0>0$ sufficiently small, we  deduce from (\ref{121}) that 
	\begin{equation}
		\liminf_{t\rightarrow \infty }t^{1/\alpha-1/2 }\|v_{\Lambda }(t,x)\|_{L^\infty _x}\ge  \left(\frac{2-\alpha }{2\alpha \lambda _2}\right)^{1/\alpha }. \notag
	\end{equation} 
	This together with (\ref{4.2}) and  Lemma \ref{l4.3-00} yields the limit (\ref{z651}). 
	\end{proof}
\begin{proof}[\textbf{Proof of Claim \ref{c123}}]
	The key observation is that the solution decays faster when $z_+=0$:
	\begin{eqnarray}
		\|u(t,x)\|_{L^\infty _x}\lesssim t^{-3/4},\ 	\|u(t,x)\|_{L^2 _x}\lesssim t^{-1/2}\qquad  \text{for }t>T^*(\varepsilon _0).\label{655}
	\end{eqnarray}
	In fact, since $z_+=0$, it follows from (\ref{4107})--(\ref{7}) that 
	\begin{eqnarray}
		v_{\Lambda }(t,x)=-i \int_{t}^{\infty }e^{iw(x)(t-s)+i\lambda \left(\Phi(t,x)-\Phi(s,x)\right)}s^{-\alpha /2} R(v)(s)\mathrm{d}s.\label{653}
	\end{eqnarray}
	On the other hand, using (\ref{eq1}) and Lemma \ref{l5.1}, we have, for $s>t>T^*(\varepsilon _0)$ 
	\begin{eqnarray}
		\|\Phi(t,x)-\Phi(s,x)\|_{L^\infty _x}\le \int_{t}^{s}\tau^{-\alpha /2}\|v_{\Lambda }(\tau,x)\|_{L^\infty _x}^{\alpha } \mathrm{d}\tau\le K_0^\alpha  \log \frac{s}{t}. \label{654}
	\end{eqnarray} 
	Substituting (\ref{654}) to (\ref{653}), and using  the  $L^\infty $ bound of  $R(v)$ in  (\ref{r2}),  we get 
	\begin{eqnarray}
		\|v_{\Lambda }(t,x)\|_{L^\infty _x}\lesssim \int_{t}^{\infty } \left(\frac{s}{t}\right)^{\lambda _2K_0^\alpha }s^{-\alpha /2}s^{-5/4+\alpha /2}\mathrm{d}s\lesssim t^{-1/4},\ t>T^*(\varepsilon _0). \notag
	\end{eqnarray}
	Similarly, we have 
	\begin{eqnarray}
		\|v_{\Lambda }(t,x)\|_{L^2 _x}\lesssim t^{-1/2},\ t>T^*(\varepsilon _0). \notag
	\end{eqnarray}
	The above two inequalities together with (\ref{4.2}) and Lemma \ref{l4.3-00} yield (\ref{655}).

	Next, we apply the decay estimates (\ref{655}) to prove that $u_0=0$. Since $z_+=0$, it follows from the equation (\ref{1.1}) and the asymptotic formula  (\ref{3302})  that  
	\begin{eqnarray}
		u(t,x)=\lambda \int_{t}^{\infty }e^{iF(D)(t-s)} (|u|^\alpha u)(s)\mathrm{d}s.\notag
	\end{eqnarray}
	Then applying  Strichartz's estimate and H\"older's inequality, we get  
	\begin{eqnarray}
		\|u\|_{L^4([T,\infty ),L_x^\infty) }&\lesssim& \int_{T}^{\infty }\|u(s,x)\|_{L^\infty _x}^\alpha \|u(s,x)\|_{L^2_x} \mathrm{d}s\notag\\
		&\lesssim &	\|u\|_{L^4([T,\infty ),L_x^\infty) }\left(\int_{T}^{\infty }\left(\|u(s,x)\|_{L^\infty _x}^{\alpha -1}\|u(s,x)\|_{L^2}\right)^{4/3} \mathrm{d}s\right)^{3/4}\notag\\
		&\le&  C\|u\|_{L^4([T,\infty ),L_x^\infty) } T^{1-3\alpha /4},\notag
	\end{eqnarray}
	where we use (\ref{655}) in the last step. Since $\alpha >4/3$, we can choose $T>T^*(\varepsilon _0)$ sufficiently large such that $CT^{1-3\alpha /4}\le \frac{1}{2}$; so that $\|u\|_{L^4([T,\infty ),L_x^\infty) }=0$. This together with the uniqueness of the solutions implies $u\equiv0$, from which Claim \ref{c123} follows. 
\end{proof}
\section*{Acknowledgements}
This work is
partially supported by  NSF of
China under Grants   11771389,  11931010 and 11621101.


\begin{thebibliography}{99}
	\bibitem{GP}
	G.P. Agrawal, Nonlinear Fiber Optics, Academic Press, Inc., 1995. 
	 
	\bibitem{Ba}
	I. Barab, 
	Nonexistence of asymptotically free solutions for nonlinear Schr\"odinger equations, 
	J. Math. Phy. 25 (1984), 3270--3273.
	

	\bibitem{CaHan}
	T. Cazenave, Z. Han,
	Asymptotic behavior for a Schr\"odinger equation with nonlinear subcritical dissipation.
	Nonlinear Anal. 205 (2021), 112243, 37pp.
	
	\bibitem{Ca2}
	T. Cazenave,  Z. Han and Y. Martel,
	Blowup on an arbitrary compact set for a Schr\"odinger equation with nonlinear source term. 
	J. Dyn. Diff. Equat. 33 (2021), no.2, 941--960. 
	
	\bibitem{CaJFA}
	T. Cazenave,  Y. Martel,
	Modified scattering for the critical nonlinear
	Schr\"odinger equation. 
	J. Funct. Anal. 274 (2018), no. 2, 402--432.
	
	\bibitem{Cr}
	M.C. Cross and P.C. Hohenberg,
	Pattern formation outside of equilibrium.
	Rev. Mod. Phys.  65 (1993), 851--1112.
	
	
	\bibitem{DZ03} P. A. Deift, X. Zhou,  Long-time asymptotics for solutions of the NLS equation with initial data in a weighted Sobolev space.  Comm. Pure Appl. Math. 56 (2003), no. 8, 1029--1077.
	
	
	
	\bibitem{Delort}
	J.M. Delort,
	Semiclassical microlocal normal forms and global solutions of modified one-dimensional KG equations.
	Ann. Inst. Fourier (Grenoble), 66 (2016), no. 4, 1451-1528.
	\bibitem{D-S}
	M. Dimassi, J. Sj\"{o}strand,
	Spectral asymptotics in the semi-classical limit.
	London Mathematical Society Lecture Note Series, 268. Cambridge University Press, Cambridge, 1999.
	
	\bibitem{Ginibre0} J. Ginibre, G. Velo, Scattering theory in the energy space for a class of nonlinear Schr\"odinger equations.  J. Math. Pures Appl. 64 (1985), 363--401. 
	
	
	
	\bibitem{Hayashi} N. Hayashi, P. Naumkin,
	Asymptotics for large time of solutions to the nonlinear Schr\"{o}dinger and Hartree equations.
	Amer. J. Math. 120 (1998), no. 2, 369-389.
	
	\bibitem{Jin}
	G. Jin, Y. Jin, C. Li,
	The initial value problem for nonlinear Schr\"odinger equations with a dissipative nonlinearity in one space dimension.
	J. Evol. Equ. 16 (2016), 983-995.
	
	
	\bibitem{Kita}
	N. Kita, A. Shimomura, 
	Asymptotic behavior of solutions to Schr\"odinger equations with a subcritical dissipative nonlinearity.
	J. Differential Equations. 242 (2007),  192-210.
	
	\bibitem{Kita2}
	N. Kita, A. Shimomura, 
	Large time behavior of  solutions to Schr\"odinger equations with a  dissipative nonlinearity for arbitrarily large initial data.
	J. Math. Soc. Japan. 61 (2009),  39-64.
	
	\bibitem{Mie}
	A. Mielke, 
	The Ginzburg-Landau equation in its role as a modulation equation, North-Holland. Amsterdam. 
	in Handbook of dynamical systems. (2002),  no. 2,  759--834.
	
	\bibitem{Ozawa}
	T. Ozawa, Long range scattering for nonlinear Schr\"odinger equations in one space dimension.
	Comm. Math. Phys. 139 (1991), 479--493.
	\bibitem{CPDE}
	A. Shimomura, 
	Asymptotic behavior of solutions for Schr\"odinger equations with dissipative nonlinearities. 
	Comm. Partial Differential Equations. 31 (2006), 1407--1423.
	
	\bibitem{Ste}
	K. Stewartson and J.T. Stuart,
	A non-linear instability theory for a wave system in plane Poiseuille flow.
	J. Fluid Mech. 48 (1971), 529-545.
	
	\bibitem{S}
	A. Stingo,
	Global existence and asymptotics for quasi-linear one-dimensional Klein-Gordon equations with mildly decaying Cauchy data.
	Bull. Soc. Math. France, 146 (2018), no. 1, 155-213.
	
	\bibitem{Strauss}
	W.A. Strauss, 
	Nonlinear Scattering Theory, Scattering Theory in Mathematical Physics, Reidel, Dordrecht, Holland, 1974, pp. 53--78, edited by J.A. Lavita and J-P. Marchand. 
	
	\bibitem{Strauss2}
	W.A. Strauss, Dispersion of low-energy waves for two conservative equations. Arch. Rat. Mech. Anal. 55 (1974),  86--92.  
	
	\bibitem{Tsutsumi}
	Y. Tsutsumi, K. Yajima, The asymptotic behavior of nonlinear Schr\"odinger equations. Bull. Am. Math. Soc. 11 (1984), 186--188. 
	\bibitem{Zhang}
	T. Zhang,
	Global solutions of modified one-dimensional Schr\"odinger equation. Commun. Math. Res. 37 (2021), no.3, 350--386. 
	
	\bibitem{Zworski}
	M. Zworski,
	Semiclassical analysis,
	Graduate Studies in Mathematics, 138. American Mathematical Society, Providence, RI, 2012.
	
\end{thebibliography}
\end{document}